\documentclass[11pt,reqno]{amsart}
\usepackage{tikz}
\usepackage{amsthm}
\usepackage{bm}
\usepackage{amsmath}
\usepackage{amsmath,amssymb}
\usepackage{subfig}
\usepackage{epstopdf}
\usepackage{epsfig}
\graphicspath{{./figures/}}
\usepackage{tikz}
\usetikzlibrary{shapes,arrows}
\tikzstyle{decision} = [diamond, draw, fill=blue!20, 
    text width=4.5em, text badly centered, node distance=3cm, inner sep=0pt]
\tikzstyle{block} = [rectangle, draw, fill=blue!20, 
    text width=5em, text centered, rounded corners, minimum height=4em]
\tikzstyle{line} = [draw, -latex']
\tikzstyle{cloud} = [draw, ellipse,fill=red!20, node distance=3cm,
    minimum height=2em]
 \newtheoremstyle{mystyle}{36pt}{}{}{}{\bfseries}{.}{ }{}
 
  \theoremstyle{plain}
 \newtheorem{theorem}{Theorem}
 \newtheorem{lemma}[]{Lemma}
 \newtheorem{definition}[]{Definition}
 
 \newtheorem{proposition}[]{Proposition}
 \newtheorem{example}{Example}

  \theoremstyle{remark}
 \newtheorem{remark}{Remark}

    \usetikzlibrary{positioning}
\tikzset{main node/.style={circle,fill=blue!20,draw,minimum size=1cm,inner sep=0pt},  }

\newcommand{\vr}{\overrightarrow}

\newcommand{\ts}{\mathsf{T}}

\newcommand{\dd}{\mathcal{\dagger}}

\begin{document}
\title[Transport information geometry]{Transport information geometry I: Riemannian calculus on probability simplex}
\author[Li]{Wuchen Li}
\email{wcli@math.ucla.edu}
\address{Department of Mathematics, University of California, Los Angeles.}
\keywords{Optimal transport; Information geometry; Probability manifold; Linear weighted Laplacian; Graph.}
\begin{abstract}
We formulate the Riemannian calculus of the probability set embedded with $L^2$-Wasserstein metric. This is an initial work of transport information geometry. 
Our investigation starts with the probability simplex (probability manifold) supported on vertices of a finite graph. The main idea is to embed the probability manifold as a submanifold of the positive measure space with a nonlinear metric tensor. Here the nonlinearity comes from the linear weighted Laplacian operator. By this viewpoint, we establish torsion--free Christoffel symbols, Levi-Civita connections, curvature tensors and volume forms in the probability manifold by Euclidean coordinates. As a consequence, the Jacobi equation, Laplace-Beltrami and Hessian operators on the probability manifold are derived. These geometric computations are also provided in the infinite-dimensional density space (density manifold) supported on a finite-dimensional manifold. In particular, an identity is given connecting the Baker-{\'E}mery $\Gamma_2$ operator (carr{\'e} du champ it{\'e}r{\'e}) by connecting Fisher-Rao information metric and optimal transport metric. Several examples are demonstrated.
\end{abstract}
\maketitle
\section{Introduction}
In recent years, optimal transport theory, a.k.a. the Monge-Kantorovich problem has attracted a lot of attention in various fields, such as partial differential equations \cite{CD, otto2001}, functional inequalities \cite{OV}, geometry \cite{CY, Lott_Villani, Mc, strum} and evolutionary dynamics \cite{li-thesis}. Given a finite-dimensional manifold $M$, it studies metrics in the set of probability measures $\mathcal{P}(M)$. In particular, the $L^2$-Wasserstein metric $W$ gives the Riemannian structures of the probability density set. In literature \cite{Lafferty}, the infinite-dimensional space $(\mathcal{P}(M), W)$, in the sense of Fr{\'e}chet manifold \cite{IM} is named density manifold. The density manifold gives the other viewpoint on many classical equations. For example, the gradient flow in density manifold connects to the Fokker-Planck equation \cite{JKO}, while the Hamiltonian flow in density manifold has deep relations with the Schr{\"o}dinger equation \cite{Carlen, Lafferty, Nelson1, Nelson2}. 

The study of the Riemannian structure of density manifold is necessary for the following three reasons. Firstly, the geometry formulas in density manifold can be used to study the long-time behaviors of the gradient \cite{CMV} and Hamiltonian flows in $\mathcal{P}(M)$; Secondly, the differential structures in $\mathcal{P}(M)$ and $M$ have many interactions. For instance, the Hessian of negative Boltzmann-Shannon entropy in density manifold relates to the Ricci curvature by Bochner's formula \cite{BE}, and connects to the Yano's formula on $M$ \cite{li-theory}; Lastly, it helps to design and analyze the evolutionary dynamics in population games; see chapter 4 of \cite{li-thesis}. Besides these motivations, Villani asks the following two questions on page 444-445 of his famous book \cite{vil2008}. ``Can one define the Christoffel symbol, Laplace operator, volume form and divergence operator on density manifold, at least formally?'' In addition, ``Open problem 15.11: Find a nice formula for the Hessian of the functional in density manifold.''

In this paper, we complete answer this question in finite dimensional sample space.  We study the Riemannian structure of $L^2$-Wasserstein metric in the probability simplex supported on finite graphs \cite{li-theory}, a topic initialized in \cite{chow2012, EM1, M}, and further extend the derivations into the infinite-dimensional positive smooth density manifold. The approach follows the study of a linear weighted Laplacian operator (a concept defined in section \ref{section2}), and the construction of Shahshahani metric, also known as Fisher-Rao metric  in evolutionary dynamics \cite{Fisher-Rao1} and information geometry \cite{Amari, divergence, IG2}. 

The main result is sketched as follows. We construct a Riemannian metric tensor in the discrete positive measure space (positive orthant) $(\mathcal{M}_+, g)$, where $g$ is a positive definite metric tensor depending on the linear weighted graph Laplacian operator. In Theorem 1, we show that the probability manifold is endowed with the $L^2$-Wasserstein metric, as a submanifold of $(\mathcal{M}_+, g)$. Following the geometry structures in $(\mathcal{M}_+, g)$, we derive the corresponding ones in the probability manifold, such as torsion-free Christoffel symbol, volume form, Laplace-Beltrami, and Hessian operators by Proposition 2-9. Similar derivations are also provided in the infinite-dimensional positive density manifold by Proposition 10-18. The Hessian operator in density manifold is provided by Proposition \ref{Hess_new}, by which the connection between the metric tensor in $\mathcal{P}(M)$ and Bakry-{\'E}mery $\Gamma_2$ operator in $M$ is introduced in Proposition \ref{col}, in which the Bochner's and Yano's formulas are connected. Here we emphasize the mathematical relation between the Fisher-Rao metric and Wasserstein metric for deriving Gamma two operators. This fact becomes clear by formulating the torsion-free Wasserstein Christoffel symbol explicitly. 

In literature, the Riemannian structure of density manifold $\mathcal{P}(M)$ has been studied from the differential structures of $M$, by the works of Lafferty \cite{Lafferty}, Lott \cite{Lott} and Gigli \cite{Gigli}. In \cite{Lafferty}, guiding by the stochastic mechanics, the Riemannian and symplectic structures of density manifold are introduced by the tangent vectors in the base manifold $M$. This approach follows Moser's theorem \cite{Moser}. The density manifold is viewed as the quotient of the group of diffeomorphisms in $M$ which preserve the Riemannian volume of $M$; By the notation of \cite{otto2001}, similar geometric calculations have been done in \cite{Lott}, in which the curvature tensor in density manifold is established by the ones in $M$. \cite{Gigli} proves the regularity issues of second-order operators in $\mathcal{P}(\mathbb{R}^d)$. 
Different from the approach of diffeomorphisms in $M$, we study the geometry of probability and density manifold by the linear weighted Laplacian operator. This angle gives the geometry formulas of probability simplex supported on both finite graphs and continuous states. Besides, the canonical volume form on density manifold is asked by \cite{WD}. We introduce the one in the discrete probability manifold. 

It is worth mentioning that the idea of embedding probability manifold into positive octant is motivated by the Shahshahani metric \cite{Fisher-Rao1}, also called the Fisher-Rao metric in information geometry \cite{Amari, divergence, IG2}. And the metric tensor in probability manifold is based on the linear weighted Laplacian matrix. This operator is closely related to the osmotic diffusion considered in Nelson's stochastic mechanics \cite{Lafferty, Nelson2}. Besides, the connection between the Bakery-{\'E}mery $\Gamma_2$ operator and the metric tensor of density manifold follows the work of \cite{BE, Nelson2, OV} and the study of Yano's formula \cite{li-theory, Yano}.

The plan of this paper is as follows. In section \ref{section2}, we review the optimal transport on a graph. 
In section \ref{external}, we derive the Riemannian structure of the discrete probability simplex set. In section \ref{section5}, we introduce the associated Riemannian calculus in the density manifold. Several examples are introduced to illustrate the geometry of probability manifold in section \ref{section4}.
\section{Review on Probability manifold on graphs}\label{section2}
In this section, we carefully review the definition of $L^2$-Wasserstein metric on finite graphs. Here the results are provided by \cite{chow2012, EM1, M}. We summarize them into the matrix formulation \cite{li-theory}. 

Consider a weighted undirected finite graph $G=(V, E, \omega)$, where 
$V=\{1,2,\cdots, n\}$
is the vertex set, $E$ is the edge set, and $\omega=(\omega_{ij})_{i,j\in V}\in \mathbb{R}^{n\times n}$ is the weight of each edge with
$\omega_{ij}=\begin{cases}\omega_{ji}>0
& \textrm{if $(i,j)\in E$}\\
0 & \textrm{otherwise}
\end{cases}$. We denote the adjacent set or neighborhood of $i$ by $N(i)=\{j\in V\colon (i,j)\in E\}$. 

The probability simplex supported on all vertices of $G$ is defined by 
\begin{equation*}
\mathcal{P}(G)=\{(\rho_1,\cdots, \rho_n)\mid \sum_{i=i}^n \rho_i=1,\quad  \rho_i\geq 0\}\in \mathbb{R}^n,
\end{equation*}
where $ \rho_i$ is the discrete probability function at node $i$, whose interior is denoted by
 $$\mathcal{P}_+(G)=\{(\rho_1,\cdots, \rho_n) \mid \sum_{i=1}^n \rho_i=1,\quad  \rho_i>0 \},$$ 
 and whose boundary is given by $\partial\mathcal{P}(G)=\mathcal{P}(G)\setminus\mathcal{P}_+(G)$.

Given a potential $\Phi\in \mathbb{R}^{n}$ on a graph, a
{\em gradient vector field} $\nabla_G\Phi=(\nabla_{ij}\Phi)_{i,j\in V}\in \mathbb{R}^{n\times n}$ refers to
\begin{equation*}
\nabla_{ij}\Phi=\sqrt{\omega_{ij}}(\Phi(i)-\Phi(j)) .\end{equation*} 
Here the {\em vector field} $v=(v_{ij})_{i,j\in V}\in \mathbb{R}^{n\times n}$ on a graph is a {\em skew-symmetric matrix}:  
\begin{equation*}
v_{ij}=\begin{cases}-v_{ji} & \textrm{if $(i,j)\in E$}\\
0 & \textrm{otherwise}
\end{cases}.
\end{equation*}
For simplicity of notation, we would like to vectorize the matrix $v\in \mathbb{R}^{n\times n}$. To do so, we decompose the indirect edge set by two direct edge sets 
$E=\overrightarrow{E}\cup \overleftarrow{E}$, i.e. each edge is assigned a given orientation. By abusing the notation, we denote $v=(v_{ij})_{(i,j)\in \overrightarrow{E}}\in \mathbb{R}^{|E|}$, $\nabla\Phi=(\nabla_{ij}\Phi)_{(i,j)\in \overrightarrow{E}}\in \mathbb{R}^{|E|}$.

The divergence of $v$, $\textrm{div}(v)=(\textrm{div}(v)(i))_{i\in V}\in \mathbb{R}^{n}$, is defined by \begin{equation*}
\textrm{div}(v)(i) = -\sum_{j\in N(i)}\sqrt{w_{ij}}v_{ij}.
\end{equation*}
Above definitions introduce the discrete integration by parts 
\begin{equation*}
\sum_{i=1}^n\Phi(i)\textrm{div}(v)(i)=-\sum_{(i,j)\in \overrightarrow{E}}\omega_{ij}v_{ij}(\Phi(i)-\Phi(j)).
\end{equation*}
It is worth mentioning that the above formula does not depend on the orientation of graph $E=\overrightarrow{E}\cup \overleftarrow{E}$, i.e.
$$\sum_{(i,j)\in \overrightarrow{E}}\omega_{ij}v_{ij}(\Phi(i)-\Phi(j))=\sum_{(i,j)\in \overleftarrow{E}}\omega_{ij}v_{ji}(\Phi(j)-\Phi(i))=\frac{1}{2}\sum_{(i,j)\in{E}}\omega_{ij}v_{ij}(\Phi(i)-\Phi(j)).$$
Here the coefficient $1/2$ in front of the summation accounts for the fact that every edge in $E$ is counted twice.

A  probability weighted vector field (flux) $((\rho v)_{ij})_{(i,j)\in \vr{E}}\in \mathbb{R}^{|E|}$ is introduced
\begin{equation*}
(\rho v)_{ij}:=v_{ij}\theta_{ij}(\rho),
\end{equation*}
where $\theta_{ij}(\rho)$ represents the probability weight on $\mathrm{edge}$ $(i,j)$, defined by 
\begin{equation*}
\theta_{ij}(\rho):=\theta(\rho_i,\rho_j)=\frac{\rho_i+\rho_j}{2}.
\end{equation*}
In discrete states, there are multiple choice of $\theta_{ij}$, such as up-wind scheme \cite{li-thesis} and logarithm mean \cite{EM1, M}. Here we focus on $\theta_{ij}$ being a linear function of $\rho$.

Given two vector fields $v$, $\tilde{v}$ on a graph and $\rho \in \mathcal{P}(G)$, denote an inner product for vector fields $(\cdot,\cdot)_{\rho}\colon\mathbb{R}^{|E|}\times \mathbb{R}^{|E|}\rightarrow \mathbb{R}$ by 
\begin{equation*}
(v, \tilde{v})_ \rho:=\sum_{(i,j)\in \vr{E}} v_{ij}\tilde{v}_{ij}\theta_{ij}(\rho). 
\end{equation*}
Based on above definitions, the $L^2$-Wasserstein distance on probability set $\mathcal{P}(G)$ is defined as follows. 
\begin{definition}[$L^2$-Wasserstein metric on a graph]
Given two points $\rho^0$, $ \rho^1\in\mathcal{P}(G)$, the metric $W\colon\mathcal{P}(G)\times \mathcal{P}(G)\rightarrow \mathbb{R}$ is defined by:
\begin{equation}\label{metric_BB}
\left(W(\rho^0,\rho^1)\right)^2:=\inf_{\rho(t), v(t)} \Big\{\int_0^1(v(t), v(t))_{\rho(t)} dt~:~ \frac{d\rho}{dt}+\mathrm{div}(\rho v)=0,~\rho(0)=\rho^0,~\rho(1)=\rho^1\Big\}.
\end{equation}
\end{definition}
The variational problem \eqref{metric_BB} admits an equivalent reformulation in terms of a Riemannian tensor, as seen in the next definition. To show this point, the following linear matrix functions are needed. 
\begin{definition}[Linear Weighted Laplacian matrix]\label{def2}
Define the matrix function $L(\cdot):\mathbb{R}^n\rightarrow \mathbb{R}^{n\times n}$ by
\begin{equation*}
L(a)=D^{\ts}\Theta(a)D,\quad a=(a_i)_{i=1}^n\in \mathbb{R}^n,
\end{equation*}
where $D \in \mathbb{R}^{|E|\times n}$ is the discrete gradient operator, i.e.
\begin{equation*} 
D_{(i,j)\in \vr{E}, k\in V}=\begin{cases}
\sqrt{\omega_{ij}} & \textrm{if $i=k$}\\ 
-\sqrt{\omega_{ij}} & \textrm{if $j=k$}\\
0 & \textrm{otherwise}
\end{cases},
\end{equation*}
$-D^{\ts}\in \mathbb{R}^{n\times |E|}$ is the discrete divergence operator (oriented incidence matrix),
and $\Theta(a)\in \mathbb{R}^{|E|\times |E|}$ is a weight matrix
\begin{equation*}
\Theta(a)_{(i,j)\in \vr{E}, (k,l)\in \vr{E}}=\begin{cases}
\theta_{ij}(a)=\frac{a_i+a_j}{2} & \textrm{if $(i,j)=(k,l)\in \vr{E}$}\\ 
0 & \textrm{otherwise}
\end{cases}.
\end{equation*}
\end{definition}
\begin{lemma}[Discrete Hodge decomposition]\label{lemma_hodge}

If $\rho\in \mathcal{P}_+(G)$, the following properties hold: 
\begin{itemize}
\item[(i)] $L(\rho)$ is semi-positive matrix with a single zero eigenvalue. Denote the eigenvalue and corresponding orthonormal eigenvectors of $L(\rho)$ by $0=\lambda_0(\rho)<\lambda_{1}(\rho)\leq\cdots\leq \lambda_{n-1}(\rho)$, and $U(\rho)=(u_0, u_1(\rho),\cdots, u_{n-1}(\rho))$,
\begin{equation*}
L(\rho)=U(\rho)\begin{pmatrix}
0 & & &\\
& \lambda_{1}(\rho)& &\\
& & \ddots & \\
& & & {\lambda_{n-1}(\rho)}
\end{pmatrix}U(\rho)^{T} .
\end{equation*}
with 
\begin{equation}\label{u0}
u_0=\frac{1}{\sqrt{n}}(1,\cdots, 1)^{\ts}.
\end{equation}
\item[(ii)] For any vector field $v$ on a graph and $\rho\in \mathcal{P}_+(G)$, there exists a unique gradient vector field $\nabla_G\Phi\in \mathbb{R}^{|E|}$ on a graph, such that
\begin{equation*}
v_{ij}= \nabla_{ij}\Phi+ \Psi_{ij},\quad \textrm{div}(\rho \Psi)=0.
\end{equation*}
where $\Psi$ is a divergence free vector field w.r.t. $\rho$. In addition, 
$$\sum_{(i,j)\in \vr{E}}v_{ij}^2\theta_{ij}(\rho)=\sum_{(i,j)\in \vr{E}}[(\nabla_{ij}\Phi)^2 + \Psi_{ij}^2]\theta_{ij}(\rho).$$
\end{itemize}
\end{lemma}
\begin{proof}
The proof is a direct extension of classical graph Hodge decomposition with the probability weight function $\theta_{ij}(\rho)$. We show that there exists a unique gradient vector field $\nabla_G\Phi$, such that
$$-\textrm{div}(\rho \nabla_G \Phi)=L(\rho) \Phi.$$
Since $\rho\in \mathcal{P}_+(G)$ and the graph is connected, then 
\begin{equation*}
\Phi^\ts L(\rho)\Phi=\sum_{(i,j)\in \vr{E}}\omega_{ij}(\Phi(i)-\Phi(j))^2\theta_{ij}(\rho)=0,
\end{equation*}
this implies that value $0$ must be a simple eigenvalue of $L(\rho)$ with eigenvector $(1,\cdots, 1)^{\ts}$. Since $\textrm{div}(\rho v)\in \textrm{Ran}L(\rho)$, and $\textrm{Ker}L(\rho)=\{u_0\}$.
Thus there exists a unique solution of $\Phi$ up to constant shrift, i.e. $\nabla_G \Phi$ is unique. And $\Psi=v-\nabla_G\Phi$ satisfies $\textrm{div}(\rho \Psi)=\textrm{div}(\rho v)-\textrm{div}(\rho\nabla\Phi)=0$. Let $v_{ij}= \nabla_{ij}\Phi+ \Psi_{ij}$, where $\textrm{div}(\rho \Psi)=0$. Then 
\begin{equation*}
\begin{split}
 \sum_{(i,j)\in \vr{E}}v_{ij}^2\theta_{ij}(\rho) =&\sum_{(i,j)\in \vr{E}}[\nabla_{ij}\Phi\nabla_{ij}{\Phi}+\Psi_{ij}{\Psi}_{ij}]\theta_{ij}(\rho)+2\sum_{i=1}^n\Phi(i)\textrm{div}(\rho {\Psi})(i)\\
 =&\sum_{(i,j)\in \vr{E}}[(\nabla_{ij}\Phi)^2+(\Psi_{ij})^2]\theta_{ij}(\rho),
 \end{split}
\end{equation*}
which finishes the proof.
\end{proof} 
From Lemma \ref{lemma_hodge}, for any discrete vector field $v$, there exists an unique $\nabla_G\Phi$ and $\Psi$, such that
$$(v,v)_\rho=(\nabla_G\Phi, \nabla_G\Phi)_\rho+(\Psi,\Psi)_\rho\geq  (\nabla_G\Phi, \nabla_G\Phi)_\rho, \quad \textrm{div}(\rho v)=\textrm{div}(\rho\nabla_G\Phi).$$
Thus the metric $W$ defined in \eqref{metric_BB} is equivalent to 
\begin{equation} \label{w2}
\begin{split}
\left(W( \rho^0, \rho^1)\right)^2=&\inf_{\Phi(t)}\{\int_0^1(\nabla_G\Phi(t), \nabla_G\Phi(t))_{\rho(t)} dt ~:~\frac{d\rho}{dt}+\textrm{div}( \rho \nabla_G\Phi)=0,~ \rho(0)= \rho^0,~ \rho(1)= \rho^1\}\\
=&\inf_{\Phi(t)}\{\int_0^1\Phi^{\ts}(t)L(\rho(t))\Phi(t) dt ~:~\frac{d\rho}{dt}=L(\rho)\Phi,~ \rho(0)= \rho^0,~ \rho(1)= \rho^1\},
\end{split}
\end{equation}
where the infimum is taken over potentials $\Phi(t)\in\mathbb{R}^n$. 
\subsection{Geometry setting}
We illustrate that \eqref{w2} gives the Riemannian structure in the probability simplex. Notice that $\mathcal{P}(G)=\mathcal{P}_+(G)\cup \partial\mathcal{P}(G)$ is a manifold with boundary. For related studies on the boundary set $\partial\mathcal{P}(G)$; see \cite{G}. In this paper, we focus on the Riemannian structure in the interior set $(\mathcal{P}_+(G), g_W)$. 

Denote the tangent space at a point $\rho \in \mathcal{P}_+(G)$, 
\begin{equation*}
T_\rho\mathcal{P}_+(G)=\{(\sigma_i)_{i=1}^n\in \mathbb{R}^n\colon \sum_{i=1}^n\sigma_i=0\},
\end{equation*}
and  
\begin{equation*}
T\mathcal{P}_+(G)=\{(\rho, \sigma)\colon \rho\in \mathcal{P}_{+}(G), \sigma\in T_\rho\mathcal{P}_+(G)\}.
\end{equation*}

\begin{definition}
The inner product $g_{W}(\cdot, \cdot):C^{\infty}(T\mathcal{P}_+(G))\times C^{\infty}(T\mathcal{P}_+(G))\rightarrow C^{\infty}(\mathcal{P}_+(G))$ is defined by
\begin{equation*}
g_{W}(\sigma_1,\sigma_2):={\sigma_1}^{\ts}L(\rho)^{\dagger}\sigma_2,\quad \textrm{for any $\sigma_1,\sigma_2\in T_\rho\mathcal{P}_+(G)$},
\end{equation*}
where matrix $L(\rho)^{\dagger}$ is the pseudo-inverse of $L(\rho)$, i.e.
\begin{equation*}
L(\rho)^{\dagger}=U(\rho)\begin{pmatrix}
0 & & &\\
& \frac{1}{\lambda_{1}(\rho)}& &\\
& & \ddots & \\
& & & \frac{1}{\lambda_{n-1}(\rho)}
\end{pmatrix}U(\rho)^{\ts} .
\end{equation*}
\end{definition}
Denote $\sigma_i=L(\rho)\Phi_i$, $i=1,2$, then 
\begin{equation}\label{relation}
{\sigma_1}^{\ts}L(\rho)^{\dagger}\sigma_2=\Phi_1^{\ts}L(\rho)L(\rho)^{\dd}L(\rho)\Phi_2=\Phi_1^{\ts}L(\rho)\Phi_2,
\end{equation}
where the second equality is from the definition of pseudo-inverse, i.e. $L(\rho)L(\rho)^{\dd}L(\rho)=L(\rho)$. 
Write $\dot\rho=\frac{d\rho}{dt}$, thus \eqref{w2} can be rewritten as 
\begin{equation*}
\left(W( \rho^0, \rho^1)\right)^2=\inf_{\rho(t)\in \mathcal{P}_+(G)}\Big\{\int_0^1\dot\rho^{\ts}L(\rho)^{\dd}\dot\rho dt~\colon~ \rho(0)= \rho^0,~ \rho(1)= \rho^1\Big\}.
\end{equation*}
Following the standard time reparametrization technique, we have
\begin{equation*}
W( \rho^0, \rho^1)=\inf_{\rho(t)\in \mathcal{P}_+(G)}\Big\{\int_0^1\sqrt{\dot\rho^{\ts}L(\rho)^{\dd}\dot\rho} dt~\colon~ \rho(0)= \rho^0,~ \rho(1)= \rho^1\Big\}.
\end{equation*}
Since $L(\rho)^{\dagger}$ is positive definite, i.e. 
\begin{equation*}
\inf_{\sigma\in \mathbb{R}^n}\Big\{\sigma^{\ts}L(\rho)^{\dd}\sigma\colon \sum_{i=1}^n\sigma_i=0,\quad \sum_{i=1}^n\sigma_i^2=1\Big\}=\frac{1}{\lambda_1(\rho)}>0,
\end{equation*}
and $L(\rho)$ is smooth w.r.t. $\rho$ because $U(\rho)$, $\lambda_i(\rho)$ are smooth, then $(\mathcal{P}_+(G), g_W)$ is a $(n-1)$ dimensional Riemannian manifold, named {\em probability manifold}.

\section{Riemannian calculus on probability manifold}\label{external}
In this section, we derive the main result of this paper. We embed $(\mathcal{P}_+(G), g_W)$ as a submanifold of discrete positive measure space with a nonlinear metric tensor. Then we derive many geometry concepts in $\mathcal{P}_+(G)$, including Christoffel symbols and curvature tensors. The Jacobi equation, Laplace-Beltrami and Hessian operators are also introduced.

Consider the discrete positive measure space (positive orthant) by
\begin{equation*}
\mathcal{M}_+(G)=\{(\mu_1,\cdots, \mu_n)\mid \mu_i>0 \}\subset \mathbb{R}^n_{+}.
\end{equation*}  
It is clear that $\mathcal{P}_+(G)\subset \mathcal{M}_{+}(G)$. And \begin{equation*}
\textrm{$T_\mu\mathcal{M}_+(G)=\mathbb{R}^n$,\quad 
$T\mathcal{M}_+(G)=\{(\mu, a)\colon \mu\in \mathcal{M}_+(G),~a\in \mathbb{R}^n\}$}.
\end{equation*}
We define a Riemannian inner product on $n$-dimensional manifold $\mathcal{M}_{+}(G)$. 
\begin{definition}[Inner product in $\mathcal{M}_+(G)$]
Define the inner product $g_{\mathcal{M}}\colon C^{\infty}(T\mathcal{M}_+(G))\times C^{\infty}(T\mathcal{M}_+(G))\rightarrow C^{\infty}(\mathcal{M}_+(G))$ by
\begin{equation*}
g_{\mathcal{M}}(a_1,a_2)=a_1^{\ts}g(\mu)a_2,\quad \textrm{for any $a_1$, $a_2\in T_\mu\mathcal{M}_+(G)$,}
\end{equation*}
where 
\begin{equation*}
g(\mu)=L(\mu)^{\dagger}+u_0u_0^{\ts}\in \mathbb{R}^{n\times n},
\end{equation*}
and $u_0$ is defined in \eqref{u0}.
\end{definition}
We shall show that $(\mathcal{M}_{+}(G),g_{\mathcal{M}})$ is a smooth $n$ dimensional Riemannian manifold, and $(\mathcal{P}_+(G), g_W)$ is a $(n-1)$ dimensional submanifold of $(\mathcal{M}_+(G), g_{\mathcal{M}})$ with the induced metric. 
\begin{theorem}[Induced metric]\label{embed}
Denote a natural inclusion by 
\begin{equation*}
\imath\colon\mathcal{P}_+(G)\rightarrow \mathcal{M}_+(G),\qquad \imath(\rho)=\rho,
\end{equation*}
then $\imath$ induces a Riemannian metric $W$ on $\mathcal{P}_+(G)$ via pullback:
\begin{equation*}
g_W(\sigma_1,\sigma_2)=g_{\mathcal{M}}(\sigma_1,\sigma_2),\quad \textrm{for any $\sigma_1,\sigma_2\in T_\rho\mathcal{P}_+(G)$}.
\end{equation*}
\end{theorem}
\begin{proof}
Since $L(\mu)^{\dagger}=U(\mu)\begin{pmatrix}
0 & & &\\
& \frac{1}{\lambda_{1}(\mu)}& &\\
& & \ddots & \\
& & & \frac{1}{\lambda_{n-1}(\mu)}
\end{pmatrix}U(\mu)^{\ts}$ with $U(\mu)=[u_0, u_1(\mu),\cdots, u_{n-1}(\mu)]$, then
\begin{equation*}
g(\mu)=\sum_{i=2}^n\frac{1}{\lambda_i(\mu)}u_i(\mu)u_i(\mu)^{\ts}+u_0u_0^{\ts}=U(\mu)\begin{pmatrix}
1 & & &\\
& \frac{1}{\lambda_{1}(\mu)}& &\\
& & \ddots & \\
& & & \frac{1}{\lambda_{n-1}(\mu)}
\end{pmatrix}U(\mu)^{\ts}.
\end{equation*}
So $g(\mu)$ is a positive definite matrix and smooth w.r.t. $\mu\in \mathcal{M}_+(G)$, and $(\mathcal{M}_+(G), g_{\mathcal{M}})$ is a smooth Riemannian manifold.

For any $\rho\in\mathcal{P}_+(G)$ and $\sigma_1, \sigma_2\in T_{\rho}\mathcal{P}_+(G)$, we need to show 
$g_W(\sigma_1,\sigma_2)=\sigma_1^{\ts}g(\rho)\sigma_2$, i.e. 
\begin{equation*}
\sigma_1^{\ts}(L(\rho)^{\dagger}+u_0u_0^{\ts})\sigma_2=\sigma_1^{\ts}L(\rho)^{\dagger}\sigma_2.
\end{equation*}
Since $u_0=\frac{1}{\sqrt{n}}(1,\cdots, 1)^{\ts}$, then $u_0^{\ts}\sigma_2=\frac{1}{\sqrt{n}}\sum_{i=1}^n\sigma_{2i}=0$, thus $\sigma_1^{\ts}u_0u_0^{\ts}\sigma_2=0$, which finishes the proof.
\end{proof}

Many geometry concepts of $(\mathcal{P}_+(G), g_W)$ follow directly from the ones in $(\mathcal{M}_+(G), g_{\mathcal{M}})$. We illustrate an example by the gradient operator. 
\begin{proposition}[Gradient]
Given $\mathcal{F}(\rho)\in C^{\infty}(\mathcal{P}_+(G))$, denote its gradient operators in $(\mathcal{M}_{+}(G),g_{\mathcal{M}})$ and $(\mathcal{P}_+(G), g_W)$ by 
$\nabla_g\mathcal{F}(\rho)\in T_{\rho}\mathcal{M}_+(G)$, $\nabla_W\mathcal{F}(\rho)\in T_{\rho}\mathcal{P}_+(G)$, respectively. Then 
\begin{equation*}
\nabla_W\mathcal{F}(\rho)=-\textrm{div}(\rho \nabla_G d_\rho\mathcal{F}(\rho)),\quad d_\rho\mathcal{F}(\rho)=(\frac{\partial}{\partial\rho_i}\mathcal{F}(\rho))_{i=1}^n,
\end{equation*}
is the orthogonal projection related to $g$ on $\mathcal{M}_+(G)$, of the restriction of $\nabla_g\mathcal{F}(\rho)$ to $\mathcal{P}_+(G)$, i.e.
\begin{equation*}
\nabla_W\mathcal{F}(\rho)=\nabla_g\mathcal{F}(\rho)-g_{\mathcal{M}}(\nabla_g\mathcal{F}(\rho), u_0)u_0.
\end{equation*}
\end{proposition}
\begin{proof}
Since 
\begin{equation*}
g(\mu)^{-1}=U(\mu)\begin{pmatrix}
1 & & &\\
& {\lambda_{1}(\mu)}& &\\
& & \ddots & \\
& & & {\lambda_{n-1}(\mu)}
\end{pmatrix}U(\mu)^{\ts}=L(\mu)+u_0u_0^{\ts},
\end{equation*}
then 
\begin{equation*}
\nabla_g\mathcal{F}(\rho)=\Big(L(\rho)+u_0u_0^{\ts}\Big)d_\rho\mathcal{F}(\rho).
\end{equation*}
For any $\sigma\in T_\rho\mathcal{P}_+(G)$, notice $\sigma^Tu_0=0$ and $u_0\in \textrm{ker}L(\rho)$, then
\begin{equation*}
\sigma^{\ts}g(\rho)u_0=\sigma^{\ts}(L(\rho)^{\dagger}+u_0u_0^{\ts})u_0=\sigma^{\ts}(L(\rho)^{\dagger}u_0=0.
\end{equation*}
Thus the unit norm vector at $\rho\in\mathcal{P}_+(G)$ is a constant vector $u_0$. The orthogonal projection of $\nabla_g\mathcal{F}(\rho)$ is
\begin{equation*}
\begin{split}
\nabla_g\mathcal{F}(\rho)-\Big(u_0^{\ts}g(\rho)\nabla_g\mathcal{F}(\rho)\Big)u_0=&\Big(L(\rho)+u_0u_0^{\ts}\Big)d_\rho\mathcal{F}(\rho)-u_0u_0^{\ts}d_\rho\mathcal{F}(\rho)\\
=&L(\rho)d_\rho\mathcal{F}(\rho),
\end{split}
\end{equation*}
which finishes the proof.
\end{proof}
We next proceed the geometric computations in $(\mathcal{P}_+(G), g_W)$. We first compute the commutator $[\cdot, \cdot]_W\colon C^{\infty}(T\mathcal{P}_{+}(G))\times  C^{\infty}(T\mathcal{P}_{+}(G))\rightarrow  C^{\infty}(T\mathcal{P}_{+}(G))$.
\begin{proposition}[Commutator]\label{commu}
Given $\sigma_1$, $\sigma_2\in T_\rho\mathcal{P}_+(G)$, then 
\begin{equation*}
[\sigma_1,\sigma_2]_W=0.
\end{equation*}
\end{proposition}
\begin{proof}
Consider $\mathcal{F}\in C^\infty(\mathcal{P}_+(G))$. Since $(\sigma_1\mathcal{F})(\rho):=\frac{d}{dt}|_{t=0}\mathcal{F}(\rho+t\sigma_1)$, then 
\begin{equation*}
\begin{split}
([\sigma_1,\sigma_2]_W\mathcal{F})(\rho)=&\frac{d}{dt}|_{t=0}\frac{d}{ds}|_{s=0}\mathcal{F}(\rho+t\sigma_1+s\sigma_2)-\frac{d}{ds}|_{s=0}\frac{d}{dt}|_{t=0}\mathcal{F}(\rho+t\sigma_1+s\sigma_2)\\
=&\sigma_1^{\ts}d^2_\rho\mathcal{F}(\rho)\sigma_2-\sigma_2^{\ts}d^2_\rho\mathcal{F}(\rho)\sigma_1=0,
\end{split}
\end{equation*}
where $d^2_\rho\mathcal{F}(\rho)=(\frac{\partial^2}{\partial\rho_i\partial\rho_j}\mathcal{F}(\rho))_{1\leq i, j\leq n}$ is the second differential of $\mathcal{F}(\rho)$.
\end{proof}

We next derive the Levi-Civita connection in probability manifold. To do so, we need the following operation. 
Given two vector fields $v$, $\tilde{v}$ on a graph, denote $\cdot\circ\cdot\colon\mathbb{R}^{|E|}\times \mathbb{R}^{|E|}\rightarrow \mathbb{R}^n$ by 
\begin{equation*}
v\circ \tilde{v}:=\frac{1}{2}(\sum_{j\in N(i)}v_{ij}\tilde{v}_{ij})_{i=1}^n\in\mathbb{R}^n.
\end{equation*}

\begin{proposition}[Levi-Civita connection]\label{Levi}
$\nabla^{W}_{\cdot}\cdot\colon C^{\infty}(T_\rho\mathcal{P}_+(G))\times C^{\infty}(T_\rho\mathcal{P}_+(G))\rightarrow C^{\infty}(T_\rho\mathcal{P}_+(G))$ is defined by
\begin{equation*}
\nabla^W_{\sigma_1}\sigma_2=-\frac{1}{2}[L(\sigma_1)L(\rho)^{\dagger}\sigma_2+L(\sigma_2)L(\rho)^{\dagger}\sigma_1]+\frac{1}{2}L(\rho)\Big(\nabla_G L(\rho)^{\dagger}\sigma_1	\circ\nabla_G L(\rho)^{\dagger}\sigma_2\Big),
\end{equation*}
where 
$\big(\nabla_G L(\rho)^{\dagger}\sigma_1	\circ \nabla_G L(\rho)^{\dagger}\sigma_2\big)=\frac{1}{2}\Big(\sum_{j\in N(i)}(\nabla_{ij}L(\rho)^{\dagger}\sigma_1) (\nabla_{ij}L(\rho)^{\dagger}\sigma_2) \Big)_{i=1}^n\in \mathbb{R}^n$.
\end{proposition}
\begin{proof}
Given $\sigma_i\in \mathcal{T}_\rho\mathcal{P}_+(G)$, $i=1,2,3$, by the Koszul formula and Proposition \ref{commu}, we have  
\begin{equation*}
\begin{split}
g_W(\nabla_{\sigma_1}^W\sigma_2, \sigma_3)
=&\frac{1}{2}\Big\{\sigma_1(\sigma_2^{\ts}L(\rho)^{\dagger}\sigma_3)+\sigma_2(\sigma_1^{\ts}L(\rho)^{\dagger}\sigma_3)-\sigma_3(\sigma_1^{\ts}L(\rho)^{\dagger}\sigma_2)\Big\}\\
=&\frac{1}{2}\frac{d}{dt}|_{t=0}\Big\{(\sigma_2^{\ts}L(\rho+t\sigma_1)^{\dagger}\sigma_3+\sigma_1^{\ts}L(\rho+t\sigma_2)^{\dagger}\sigma_3-\sigma_1^{\ts}L(\rho+t\sigma_3)^{\dagger}\sigma_2\Big\}\\
=&-\frac{1}{2}\Big[\sigma_2^{\ts}L(\rho)^{\dagger}L(\sigma_1)L(\rho)^{\dagger}\sigma_3+\sigma_1^{\ts}L(\rho)^{\dagger}L(\sigma_2)L(\rho)^{\dagger}\sigma_3\Big]+\frac{1}{2}\sigma_1^{\ts}L(\rho)^{\dagger}L(\sigma_3)L(\rho)^{\dagger}\sigma_2\\
=&\sigma_3^{\ts}L(\rho)^{\dagger}\Big(-\frac{1}{2}[L(\sigma_1)L(\rho)^{\dagger}\sigma_2+L(\sigma_2)L(\rho)^{\dagger}\sigma_1]+\frac{1}{2}L(\rho)(\nabla_G L(\rho)^{\dagger}\sigma_1\circ \nabla_G L(\rho)^{\dagger}\sigma_2)\Big),
\end{split}
\end{equation*}
where the last two equalities are from the Claim 1 and 2 proved in below. Since $g_W(\nabla_{\sigma_1}^W\sigma_2, \sigma_3)=(\nabla_{\sigma_1}^W\sigma_2)^{\ts}L(\rho)^{\dagger}\sigma_3$, for any $\sigma_3\in \mathcal{P}_+(G)$, the result is proved.

\textbf{Claim 1:}
\begin{equation*}
\frac{d}{dt}|_{t=0}\sigma_2^{\ts}L(\rho+t\sigma_1)^{\dagger}\sigma_3=-\sigma_2^{\ts}L(\rho)^{\dagger}L(\sigma_1)L(\rho)^{\dagger}\sigma_3.
\end{equation*}
\begin{proof}[Proof of Claim 1]
From Lemma \ref{embed}, $\sigma_2^{\ts}L(\rho+t\sigma_1)^{\dagger}\sigma_3=\sigma_2^{\ts}g(\rho+t\sigma_1)\sigma_3$. Since $L(\rho)=-D^{\ts}\Theta(\rho)D$ is linear on $\rho$, we can show that
\begin{equation}\label{claim1}
\frac{d}{dt}g(\rho+t\sigma_1)=-g(\rho+t\sigma_1)L(\sigma_1)g(\rho+t\sigma_1).
\end{equation}
This is true since 
\begin{equation}\label{claim1_step}
\begin{split}
\frac{d}{dt}g(\rho+t\sigma_1)=&-g(\rho+t\sigma_1)\cdot\frac{d}{dt}g(\rho+t\sigma_1)^{\dagger}\cdot g(\rho+t\sigma_1)\\
=&-g(\rho+t\sigma_1)\cdot\frac{d}{dt}[L(\rho+t\sigma_1)+u_0u_0^{\ts}]\cdot g(\rho+t\sigma_1)\\ 
=&-g(\rho+t\sigma_1)\cdot L(\sigma_1)\cdot g(\rho+t\sigma_1),
\end{split}
\end{equation}
where the first equality holds since 
\begin{equation*}
0=\frac{d}{dt}\mathbb{I}=\frac{d}{dt}(g(\rho+t\sigma_1)\cdot g(\rho+t\sigma_1)^{-1})=\frac{d}{dt}g(\rho+t\sigma_1)\cdot g(\rho+t\sigma_1)^{-1}+g(\rho+t\sigma_1)\cdot \frac{d}{dt}g(\rho+t\sigma_1)^{-1},
\end{equation*}
and the second equality is true because the element in matrix function $L(\cdot)$ is linear.
Hence from \eqref{claim1},
\begin{equation*}
\begin{split}
\frac{d}{dt}\sigma_2^{\ts}L(\rho+t\sigma_1)^{\dagger}\sigma_3=&\sigma_2^{\ts}\frac{d}{dt}|_{t=0}g(\rho+t\sigma_1)\sigma_3\\
=&-(g(\rho)\sigma_2)^{\ts}\cdot L(\sigma_1)\cdot g(\rho)\sigma_3.
\end{split}
\end{equation*}
For any $\sigma\in T_\rho\mathcal{P}_+(G)$, then
\begin{equation*}
g(\rho)\sigma=L(\rho)^{\dagger}\sigma+u_0u_0^{\ts}\sigma=L(\rho)^{\dagger}\sigma+(\frac{1}{n}\sum_{i=1}^n\sigma_i)u_0=L(\rho)^{\dagger}\sigma.
\end{equation*}
Putting this in \eqref{claim1_step}, we prove the claim 1.
\end{proof}
\textbf{Claim 2:}
\begin{equation*}
\sigma_1^{\ts}L(\rho)^{\dagger}L(\sigma_3)L(\rho)^{\dagger}\sigma_2=\sigma_3^{\ts}L(\rho)^{\dagger}L(\rho)\Big(\nabla_G L(\rho)^{\dagger}\sigma_1\circ \nabla_G L(\rho)^{\dagger}\sigma_2\big).
\end{equation*}
\begin{proof}[Proof of Claim 2]
Since $L(\sigma_3)=D^{\ts}\Theta(\sigma_3)D$, then
\begin{equation*}
\begin{split}
&\sigma_1^{\ts}L(\rho)^{\dagger}L(\sigma_3)L(\rho)^{\dagger}\sigma_2\\
=&(DL(\rho)^{\dagger}\sigma_1)^{\ts}\Theta(\sigma_3)DL(\rho)^{\dagger}\sigma_2\\
=&\frac{1}{2}\sum_{(i,j)\in E}(\nabla_{ij}L(\rho)^{\dagger}\sigma_1)(\nabla_{ij}L(\rho)^{\dagger}\sigma_2)\frac{\sigma_3(i)+\sigma_3(j)}{2}\\
=&\frac{1}{2}\sum_{i=1}^n\sigma_2(i)\frac{1}{2}\sum_{j\in N(i)}(\nabla_{ij}L(\rho)^{\dagger}\sigma_1)(\nabla_{ij}L(\rho)^{\dagger}\sigma_2)+\frac{1}{2}\sum_{j=1}^n\sigma_3(j)\frac{1}{2}\sum_{i\in N(j)}(\nabla_{ij}L(\rho)^{\dagger}\sigma_1)(\nabla_{ij}L(\rho)^{\dagger}\sigma_2)\\
=&\frac{1}{4}\sigma_3^{\ts}\Big(\sum_{j\in N(i)}(\nabla_{ij} L(\rho)^{\dagger}\sigma_1)(\nabla_{ij} L(\rho)^{\dagger}\sigma_2)\Big)_{i=1}^n+\frac{1}{4}\sigma_3^{\ts}\Big(\sum_{i\in N(j)}(\nabla_{ij} L(\rho)^{\dagger}\sigma_1)(\nabla_{ij} L(\rho)^{\dagger}\sigma_2)\Big)_{j=1}^n\\
=&\frac{1}{2}\sigma_3^{\ts}\Big(\sum_{j\in N(i)}(\nabla_{ij} L(\rho)^{\dagger}\sigma_1)(\nabla_{ij} L(\rho)^{\dagger}\sigma_2)\Big)_{i=1}^n.
\end{split}
\end{equation*}
where the second last equality holds by relabeling $i$ and $j$, i.e.
\begin{equation*}
\Big(\sum_{j\in N(i)}(\nabla_{ij} L(\rho)^{\dagger}\sigma_1)(\nabla_{ij} L(\rho)^{\dagger}\sigma_2)\Big)_{i=1}^n=\Big(\sum_{i\in N(j)}(\nabla_{ij} L(\rho)^{\dagger}\sigma_1)(\nabla_{ij} L(\rho)^{\dagger}\sigma_2)\Big)_{j=1}^n.
\end{equation*}
Since $(\nabla_GL(\rho)^{\dagger}\sigma_1\circ \nabla_GL(\rho)^{\dagger}\sigma_2)=\frac{1}{2}\Big(\sum_{j\in N(i)}\nabla_{ij} L(\rho)^{\dagger}\sigma_1\cdot\nabla_{ij} L(\rho)^{\dagger}\sigma_2\Big)_{i=1}^n\in \mathbb{R}^n$, then
\begin{equation*}
\begin{split}
&\frac{1}{2}\sigma_3^{\ts}\Big(\sum_{j\in N(i)}\nabla_{ij} L(\rho)^{\dagger}\sigma_1\nabla_{ij} L(\rho)^{\dagger}\sigma_2\Big)_{i=1}^n\\
=&\sigma_3^{\ts}g(\rho)g(\rho)^{-1}(\nabla_GL(\rho)^{\dagger}\sigma_1\circ \nabla_GL(\rho)^{\dagger}\sigma_2)\\
=&\sigma_3^{\ts}(L(\rho)+u_0u_0^{\ts})(L(\rho)^{\dagger}+u_0u_0^{\ts})(\nabla_GL(\rho)^{\dagger}\sigma_1\circ \nabla_GL(\rho)^{\dagger}\sigma_2)\\
=&\sigma_3^{\ts}L(\rho)L(\rho)^{\dagger}(\nabla_GL(\rho)^{\dagger}\sigma_1\circ \nabla_GL(\rho)^{\dagger}\sigma_2),
\end{split}
\end{equation*}
which finishes the proof.
\end{proof}
\end{proof}

By the above derived Levi-Civita connection, we introduce the Christoffel symbol of $(\mathcal{P}_+(G), g_W)$ by the standard Euclidean basis, i.e. $(\frac{\partial}{\partial \rho_1},\cdots, \frac{\partial}{\partial\rho_n})$. 
\begin{proposition}[Christoffel symbol] 
Let $\Gamma^{W,k}=(\Gamma^{W,k}_{ij})_{1\leq i,j\leq n}\in \mathbb{R}^{n\times n}$, such that 
\begin{equation*}
\sum_{1\leq i,j\leq n}\Gamma_{ij}^{W,k}\sigma_1(i)\sigma_2(j)=(\nabla_{\sigma_1}^W\sigma_2)(k).
\end{equation*}
 In other words, 
 \begin{equation*}
\begin{split}
\Gamma_{ij}^{W,k}=&\frac{1}{4}\sum_{k'\in N(k)}\omega_{kk'}\Big\{(L_{ki}^{\dd}-L_{k'i}^{\dd})(\delta_{jk}+\delta_{jk'})+(L_{kj}^{\dd}-L_{k'j}^{\dd})(\delta_{ik}+\delta_{ik'}) \\
&\hspace{2cm}+\theta_{kk'}\sum_{k''\in N(k)}\omega_{kk''}(L_{ki}^{\dd}-L_{k''i}^{\dd})(L_{kj}^{\dd}-L_{k''j}^{\dd})\\
&\hspace{2cm}-\theta_{kk'}\sum_{k'''\in N(k')}\omega_{k'k'''}(L_{k'i}^{\dd}-L_{k'''i}^{\dd})(L_{k'j}^{\dd}-L_{k'''j}^{\dd})\Big\},
\end{split}
\end{equation*}
where $L(\rho)^{\dagger}=(L^{\dagger}_{ij})_{1\leq i,j\leq n}$ and $\delta_{ij}=\begin{cases} 1 & \textrm{if $i=j$}\\
0& \textrm{if $i\neq j$}
\end{cases}
$. 
Here we use the neighborhood relation on graph by $$k''-k-k'-k'''.$$
\end{proposition}
\begin{proof}
Denote $\Phi_i:=L(\rho)^{\dd}\sigma_i\in \mathbb{R}^n$, $i=1,2$, i.e. $\Phi_i(k)=\sum_{j=1}^nL_{kj}^{\dd}\sigma_j$. And 
\begin{equation}\label{Phi12}
\Phi_{12}:=\frac{1}{2}(\sum_{k'\in N(k)}\nabla_{kk'}\Phi_1\nabla_{kk'}\Phi_2)_{k=1}^n\in \mathbb{R}^n.
\end{equation}
From Proposition \ref{Levi}, 
\begin{equation*}
\begin{split}
\nabla^W_{\sigma_1}\sigma_2(k)=&-\frac{1}{2}\big[L(\sigma_1)L(\rho)^{\dd}\sigma_2+L(\sigma_2)L(\rho)^{\dd}\sigma_1-L(\rho)(\nabla_GL(\rho)^{\dd}\sigma_1\circ \nabla_GL(\rho)^{\dd}\sigma_2)\big](k)\\
=&-\frac{1}{2}\sum_{k'\in N(k)}(\Phi_2(k)-\Phi_2(k'))\frac{\Phi_1(k)+\Phi_1(k')}{2} \hspace{0.8cm} (C1)\\
&-\frac{1}{2}\sum_{k'\in N(k)}(\Phi_1(k)-\Phi_1(k'))\frac{\Phi_2(k)+\Phi_2(k')}{2} \hspace{0.8cm} (C2)\\
&+\frac{1}{2}\sum_{k'\in N(k)}(\Phi_{12}(k)-\Phi_{12}(k'))\theta_{kk'}\Big\}\hspace{2.2cm} (C3)
\end{split}
\end{equation*}
We next compute (C1), (C2), (C3). Here 
\begin{equation*}
\begin{split}
(C1)=&-\frac{1}{2}\sum_{k'\in N(k)}\omega_{kk'}\sum_{j=1}^n(L_{kj}^{\dd}-L_{k'j}^{\dd})\sigma_2(j)\frac{\sigma_1(k)+\sigma_1(k')}{2} \\
=&\frac{1}{4}\sum_{k'\in N(k)}\omega_{kk'}\sum_{j=1}^n(L_{k'j}^{\dd}-L_{kj}^{\dd})\sigma_2(j)\sum_{i=1}^n(\delta_{ki}+\delta_{k'i})\sigma_1(i)\\
=&\frac{1}{4}\sum_{i=1}^n\sum_{j=1}^n\sum_{k'\in N(k)}\omega_{kk'}(L_{k'j}^{\dd}-L_{kj}^{\dd})(\delta_{ki}+\delta_{k'i})\sigma_1(i)\sigma_2(j).
\end{split}
\end{equation*}
Similarly, 
\begin{equation*}
(C2)=\frac{1}{4}\sum_{i=1}^n\sum_{j=1}^n\sum_{k'\in N(k)}\omega_{kk'}(L_{k'i}^{\dd}-L_{ki}^{\dd})(\delta_{kj}+\delta_{k'j})\sigma_1(i)\sigma_2(j).
\end{equation*}
And recall $\Phi_{12}$ defined in \eqref{Phi12}, then
\begin{equation*}
\begin{split}
(C3)=&\frac{1}{2}\sum_{k'\in N(k)}\theta_{kk'}\omega_{kk'}(\Phi_{12}(k)-\Phi_{12}(k'))\\
=&\frac{1}{4}\sum_{k'\in N(k)}\theta_{kk'}\omega_{kk'}\Big\{ \sum_{k''\in N(k)}\omega_{kk''}(\Phi_1(k)-\Phi_1(k''))(\Phi_2(k)-\Phi_2(k'')) \\
&\hspace{2.5cm}- \sum_{k'''\in N(k')}\omega_{k'k'''}(\Phi_1(k')-\Phi_1(k'''))(\Phi_2(k')-\Phi_2(k''')) \Big\}\\
=&\frac{1}{4}\sum_{k'\in N(k)}\theta_{kk'}\omega_{kk'}\Big\{ \sum_{k''\in N(k)}\omega_{kk''}\sum_{j=1}^n(L_{jk}^{\dd}-L_{jk''}^{\dd})\sigma_1(j)\sum_{i=1}^n(L_{ik}^{\dd}-L_{ik''}^{\dd})\sigma_2(i) \\
&\hspace{2.5cm}- \sum_{k'''\in N(k')}\omega_{k'k'''}\sum_{j=1}^n(L_{jk'}^{\dd}-L_{jk'''}^{\dd})\sigma_1(j)\sum_{i=1}^n(L_{ik'}^{dd}-L_{ik'''}^{\dd})\sigma_2(i)\Big\}\\
=&\frac{1}{4}\sum_{i=1}^n\sum_{j=1}^n\sum_{k'\in N(k)}\theta_{kk'}\omega_{kk'}\Big\{ \sum_{k''\in N(k)}\omega_{kk''}(L_{jk}^{\dd}-L_{jk''}^{\dd})(L_{ik}^{\dd}-L_{ik''}^{\dd}) \\
&\hspace{3.8cm}- \sum_{k'''\in N(k')}\omega_{k'k'''}(L_{jk'}^{\dd}-L_{jk'''}^{\dd})(L_{ik'}^{\dd}-L_{ik'''}^{\dd})\Big\}\sigma_1(j)\sigma_2(i).
\end{split}
\end{equation*}
From $\sum_{1\leq i,j\leq n}\Gamma_{ij}^{W,k}\sigma_1(i)\sigma_2(j)=(\nabla_{\sigma_1}^W\sigma_2)(k)=(C1)+(C2)+(C3)$, we prove the result.

\end{proof}

\begin{proposition}[Parallel transport]
Denote $\rho\colon (a,b)\rightarrow \mathcal{P}_+(G)$ be a smooth curve. Consider $\sigma(t)\in \mathcal{P}_+(G)$ be a vector field along curve $\rho$, then the equation for $\sigma(t)$ to be parallel along $\rho(t)$ is
\begin{equation*}
\dot\sigma-\frac{1}{2}(L(\sigma)L(\rho)^{\dagger}\dot\rho+L(\dot\rho)L(\rho)^{\dagger}\sigma)+\frac{1}{2}L(\rho)\Big(\nabla_GL(\rho)^{\dagger}\dot\rho\circ\nabla_GL(\rho)^{\dagger}\sigma\Big)=0.
\end{equation*}
The geodesic equation satisfies
\begin{equation}\label{geo}
\ddot\rho-L(\dot\rho)L(\rho)^{\dagger}\dot\rho+\frac{1}{2}L(\rho)\big(\nabla_GL(\rho)^{\dagger}\dot\rho\circ\nabla_GL(\rho)^{\dagger}\dot\rho\big)=0.
\end{equation}
\end{proposition}
\begin{proof}
The parallel equation is derived by computing $\nabla_{\dot\rho(t)}^W\sigma(t)=\Big(\dot\sigma_k+\sum_{1\leq i,j\leq n}\Gamma_{ij}^{W,k}\sigma_i\dot\rho_j\Big)_{k=1}^n=0$. 
Let $\sigma(t)=\dot\rho(t)$, then the geodesic is derived by setting $\nabla^W_{\dot\rho(t)}\dot\rho(t)=(\ddot\rho_k+\sum_{1\leq i,j\leq n}\Gamma_{ij}^{W,k}\dot\rho_i\dot\rho_j)_{k=1}^n=0$.
\end{proof}

We are ready to give the curvature tensor, $R_{W}(\cdot, \cdot)\cdot\colon C^{\infty}(T\mathcal{P}_+(G))\times  C^{\infty}(T\mathcal{P}_+(G))\times  C^{\infty}(T\mathcal{P}_+(G))\rightarrow  C^{\infty}(T\mathcal{P}_+(G))$.
\begin{proposition}[Curvature tensor]
Given $\sigma_1$, $\sigma_2$, $\sigma_3$, $\sigma_4\in T_\rho\mathcal{P}_+(G)$, then 
\begin{equation*}
\begin{split}
&g_W(R^W(\sigma_1,\sigma_2)\sigma_3, \sigma_4)\\
=&\frac{1}{4}\Big\{\sigma_2^{\ts}L(\rho)^{\dagger} L(m(\sigma_1, \sigma_3))L(\rho)^{\dagger}\sigma_4 + \sigma_1^{\ts}L(\rho)^{\dagger}L(m(\sigma_2, \sigma_4))L(\rho)^{\dagger}\sigma_3\\
&-\sigma_2^{\ts}L(\rho)^{\dagger} L(m(\sigma_1, \sigma_4))L(\rho)^{\dagger}\sigma_3- \sigma_1^{\ts}L(\rho)^{\dagger} L(m(\sigma_2, \sigma_3))L(\rho)^{\dagger}\sigma_4\\
&+2 n(\sigma_1, \sigma_2)^{\ts}L(\rho)^{\dagger}n(\sigma_3,\sigma_4)+n(\sigma_1,\sigma_3)^{\ts}L(\rho)^{\dagger}n(\sigma_2,\sigma_4)-n(\sigma_2,\sigma_3)^{\ts}L(\rho)^{\dagger}n(\sigma_1,\sigma_4)\Big\},
\end{split}
\end{equation*}
where $m$, $n\colon T_\rho\mathcal{P}_+(G)\times T_\rho\mathcal{P}_+(G)\rightarrow T_\rho\mathcal{P}_+(G)$ are symmetric, antisymmetric operators defined by
\begin{equation*}
m(\sigma_a, \sigma_b):=-[L(\sigma_a)L(\rho)^{\dagger}\sigma_b+L(\sigma_b)L(\rho)^{\dagger}\sigma_a ]+\frac{1}{2}L(\rho)(\nabla_GL(\rho)^{\dagger}\sigma_a\circ\nabla_GL(\rho)^{\dagger}\sigma_b),
\end{equation*}
and 
\begin{equation*}
n(\sigma_a, \sigma_b):=L(\sigma_a)L(\rho)^{\dagger}\sigma_b-L(\sigma_b)L(\rho)^{\dagger}\sigma_a.
\end{equation*}
\end{proposition}
\begin{proof}
Since $[\sigma_1,\sigma_2]_W=0$, then $R^W(\sigma_1,\sigma_2)\sigma_3=\nabla_{\sigma_1}^W\nabla_{\sigma_2}^W\sigma_3-\nabla_{\sigma_2}^W\nabla_{\sigma_1}^W\sigma_3$. 
Thus 
\begin{equation}\label{6main}
\begin{split}
g_W(R^W(\sigma_1,\sigma_2)\sigma_3, \sigma_4)=&\sigma_1( g_W(\nabla_{\sigma_2}^W\sigma_3, \sigma_4))-g_W(\nabla_{\sigma_2}^W\sigma_3, \nabla_{\sigma_1}^W\sigma_4)\\
-&\sigma_2( g_W(\nabla_{\sigma_1}^W\sigma_3, \sigma_4))+g_W(\nabla_{\sigma_1}^W\sigma_3, \nabla_{\sigma_2}^W\sigma_4).
\end{split}
\end{equation}

We first compute $\sigma_1( g_W(\nabla_{\sigma_2}^W\sigma_3, \sigma_4))$. Denote $\dot\rho(t)=\sigma_1$ and notice Claim 1, then
\begin{equation}\label{6-1}
\begin{split}
&\sigma_1( g_W(\nabla_{\sigma_2}^W\sigma_3, \sigma_4))\\
=&\frac{d}{dt}|_{t=0}\Big\{-\frac{1}{2}\Big[\sigma_3^{\ts}L(\rho(t))^{\dagger}L(\sigma_2)L(\rho(t))^{\dagger}\sigma_4+\sigma_2^{\ts}L(\rho(t))^{\dagger}L(\sigma_3)L(\rho(t))^{\dagger}\sigma_4\Big]\\
&\hspace{1.25cm}+\frac{1}{2}\sigma_2^{\ts}L(\rho(t))^{\dagger}L(\sigma_4)L(\rho(t))^{\dagger}\sigma_3\Big\}\\
=&\frac{1}{2}\Big[\sigma_3^{\ts}L(\rho)^{\dagger}L(\sigma_1)L(\rho)^{\dagger}L(\sigma_2)L(\rho)^{\dagger}\sigma_4+\sigma_3^{\ts}L(\rho)^{\dagger}L(\sigma_2)L(\rho)^{\dagger}L(\sigma_1)L(\rho)^{\dagger}\sigma_4\\
&+\sigma_2^{\ts}L(\rho)^{\dagger}L(\sigma_1)L(\rho)^{\dagger}L(\sigma_3)L(\rho)^{\dagger}\sigma_4+\sigma_2^{\ts}L(\rho)^{\dagger}L(\sigma_3)L(\rho)^{\dagger}L(\sigma_1)L(\rho)^{\dagger}\sigma_4\Big]\\
-&\frac{1}{2}\Big[\sigma_2^{\ts}L(\rho)^{\dagger}L(\sigma_1)L(\rho)^{\dagger}L(\sigma_4)L(\rho)^{\dagger}\sigma_3+\sigma_2^{\ts}L(\rho)^{\dagger}L(\sigma_4)L(\rho)^{\dagger}L(\sigma_1)L(\rho)^{\dagger}\sigma_3\Big].
\end{split}
\end{equation}
Similarly, 
\begin{equation}\label{6-2}
\begin{split}
&\sigma_2( g_W(\nabla_{\sigma_1}^W\sigma_3, \sigma_4))\\
=&\frac{1}{2}\Big[\sigma_3^{\ts}L(\rho)^{\dagger}L(\sigma_2)L(\rho)^{\dagger}L(\sigma_1)L(\rho)^{\dagger}\sigma_4+\sigma_3^{\ts}L(\rho)^{\dagger}L(\sigma_1)L(\rho)^{\dagger}L(\sigma_2)L(\rho)^{\dagger}\sigma_4\\
&+\sigma_1^{\ts}L(\rho)^{\dagger}L(\sigma_2)L(\rho)^{\dagger}L(\sigma_3)L(\rho)^{\dagger}\sigma_4+\sigma_1^{\ts}L(\rho)^{\dagger}L(\sigma_3)L(\rho)^{\dagger}L(\sigma_2)L(\rho)^{\dagger}\sigma_4\Big]\\
-&\frac{1}{2}\Big[\sigma_1^{\ts}L(\rho)^{\dagger}L(\sigma_2)L(\rho)^{\dagger}L(\sigma_4)L(\rho)^{\dagger}\sigma_3+\sigma_1^{\ts}L(\rho)^{\dagger}L(\sigma_4)L(\rho)^{\dagger}L(\sigma_2)L(\rho)^{\dagger}\sigma_3\Big].
\end{split}
\end{equation}

We next derive $\sigma_2( g_W(\nabla_{\sigma_1}^W\sigma_3, \sigma_4))$. From proposition \ref{Levi}, we have
\begin{equation*}
\begin{split}
&g_W(\nabla_{\sigma_2}^W\sigma_3, \nabla_{\sigma_1}^W\sigma_4)\\
=&\frac{1}{2}\Big\{\sigma_2^{\ts}L(\rho)^{\dagger}L(\nabla_{\sigma_1}^W\sigma_4)L(\rho)^{\dagger}\sigma_3-\sigma_2^{\ts}L(\rho)^{\dagger}L(\sigma_3)L(\rho)^{\dagger}\nabla_{\sigma_1}^W\sigma_4-\sigma_3^{\ts}L(\rho)^{\dagger}L(\sigma_{2})L(\rho)^{\dagger}\nabla_{\sigma_1}^W\sigma_4\Big\}.
\end{split}
\end{equation*}
Notice 
\begin{equation*}
\nabla_{\sigma_1}^W\sigma_4=\frac{1}{2}\{L(\rho)(\nabla_GL(\rho)^{\dagger}\sigma_1\circ \nabla_GL(\rho)^{\dagger}\sigma_4)-L(\sigma_4)L(\rho)^{\dagger}\sigma_1-L(\sigma_1)L(\rho)^{\dagger}\sigma_4\}
\end{equation*}
and $L(\cdot)$ is a matrix function linear on $\rho$. Thus
\begin{equation}\label{6-3}
\begin{split}
&g_W(\nabla_{\sigma_2}^W\sigma_3, \nabla_{\sigma_1}^W\sigma_4)\\
=&\frac{1}{4}\Big\{\frac{1}{2}\sigma_2^{\ts}L(\rho)^{\dagger}L\Big(L(\rho)(\nabla_GL(\rho)^{\dagger}\sigma_1\circ \nabla_GL(\rho)^{\dagger}\sigma_4)\Big)L(\rho)^{\dagger}\sigma_3\\
&+\frac{1}{2}\sigma_1^{\ts}L(\rho)^{\dagger}L\Big(L(\rho)(\nabla_GL(\rho)^{\dagger}\sigma_2\circ \nabla_GL(\rho)^{\dagger}\sigma_3)\Big)L(\rho)^{\dagger}\sigma_4\\
&-\sigma_2^{\ts}L(\rho)^{\dagger}L\Big(L(\sigma_1)L(\rho)^{\dagger}\sigma_4+L(\sigma_4)L(\rho)^{\dagger}\sigma_1\Big)L(\rho)^{\dagger}\sigma_3\\
&-\sigma_1^{\ts}L(\rho)^{\dagger}L(L(\sigma_2)L(\rho)^{\dagger}\sigma_3+L(\sigma_3)L(\rho)^{\dagger}\sigma_2)L(\rho)^{\dagger}\sigma_4\\
&+\sigma_2^{\ts}L(\rho)^{\dagger}L(\sigma_3)L(\rho)^{\dagger}L(\sigma_4)L(\rho)^{\dagger}\sigma_1+\sigma_2^{\ts}L(\rho)^{\dagger}L(\sigma_3)L(\rho)^{\dagger}L(\sigma_1)L(\rho)^{\dagger}\sigma_4\\
&+\sigma_3^{\ts}L(\rho)^{\dagger}L(\sigma_2)L(\rho)^{\dagger}L(\sigma_4)L(\rho)^{\dagger}\sigma_1+\sigma_3^{\ts}L(\rho)^{\dagger}L(\sigma_2)L(\rho)^{\dagger}L(\sigma_1)L(\rho)^{\dagger}\sigma_4\Big\}.
\end{split}
\end{equation}
In the derivation of \eqref{6-3}, we use the following results, which can be proved similarly as the ones in Claim 2:
\begin{equation*}
\sigma_2^{\ts}L(\rho)^{\dd}L(\sigma_3)\Big(\nabla_GL(\rho)^{\dd}\sigma_1\circ\nabla_GL(\rho)^{\dd}\sigma_4)\Big)=\sigma_1^{\ts}L(\rho)^{\dd}L\Big(L(\sigma_3)L(\rho)^{\dd}\sigma_2\Big)L(\rho)^{\dd}\sigma_4,
\end{equation*}
and
\begin{equation*}\label{6-mid}
\begin{split}
&\sigma_2^{\ts}L(\rho)^{\dd}L\Big(L(\rho)(\nabla_GL(\rho)^{\dd}\sigma_1\circ\nabla_GL(\rho)^{\dd}\sigma_4)\Big)L(\rho)^{\dd}\sigma_3\\
=&\sigma_1^{\ts}L(\rho)^{\dd}L\Big(L(\rho)(\nabla_GL(\rho)^{\dd}\sigma_2\circ\nabla_GL(\rho)^{\dd}\sigma_3)\Big)L(\rho)^{\dd}\sigma_4\\
=&(L(\rho)^{\dd}\sigma_2\circ L(\rho)^{\dd}\sigma_3)^{\ts} L(\rho)   (\nabla_GL(\rho)^{\dd}\sigma_1\circ\nabla_GL(\rho)^{\dd}\sigma_4).
\end{split}
\end{equation*}
Similarly, 
\begin{equation}\label{6-4}
\begin{split}
&g_W(\nabla_{\sigma_1}^W\sigma_3, \nabla_{\sigma_2}^W\sigma_4)\\
=&\frac{1}{4}\Big\{\frac{1}{2}\sigma_1^{\ts}L(\rho)^{\dagger}L\Big(L(\rho)(\nabla_GL(\rho)^{\dagger}\sigma_2\circ \nabla_GL(\rho)^{\dagger}\sigma_4)\Big)L(\rho)^{\dagger}\sigma_3\\
&+\frac{1}{2}\sigma_2^{\ts}L(\rho)^{\dagger}L\Big(L(\rho)(\nabla_GL(\rho)^{\dagger}\sigma_1\circ \nabla_GL(\rho)^{\dagger}\sigma_3)\Big)L(\rho)^{\dagger}\sigma_4\\
&-\sigma_1^{\ts}L(\rho)^{\dagger}L\Big(L(\sigma_2)L(\rho)^{\dagger}\sigma_4+L(\sigma_4)L(\rho)^{\dagger}\sigma_2\Big)L(\rho)^{\dagger}\sigma_3\\
&-\sigma_2^{\ts}L(\rho)^{\dagger}L(L(\sigma_1)L(\rho)^{\dagger}\sigma_3+L(\sigma_3)L(\rho)^{\dagger}\sigma_1)L(\rho)^{\dagger}\sigma_4\\
&+\sigma_1^{\ts}L(\rho)^{\dagger}L(\sigma_3)L(\rho)^{\dagger}L(\sigma_4)L(\rho)^{\dagger}\sigma_2+\sigma_1^{\ts}L(\rho)^{\dagger}L(\sigma_3)L(\rho)^{\dagger}L(\sigma_2)L(\rho)^{\dagger}\sigma_4\\
&+\sigma_3^{\ts}L(\rho)^{\dagger}L(\sigma_1)L(\rho)^{\dagger}L(\sigma_4)L(\rho)^{\dagger}\sigma_2+\sigma_3^{\ts}L(\rho)^{\dagger}L(\sigma_1)L(\rho)^{\dagger}L(\sigma_2)L(\rho)^{\dagger}\sigma_4\Big\}.
\end{split}
\end{equation}
Substituting \eqref{6-1}, \eqref{6-2}, \eqref{6-3}, \eqref{6-4} into \eqref{6main}, we derive the result.  
\end{proof}
By the curvature tensor, we introduce all types of curvatures on probability manifold by the orthonormal basis in $\{X_1,\cdots, X_{n-1}\}\subset T_\rho\mathcal{P}_+(G)$, i.e. 
\begin{equation*}
X_i={\sqrt{\lambda_{i}(\rho)}}u_{i}(\rho),\quad i=1,\cdots, n-1.
\end{equation*}
It is clear that
\begin{equation}\label{orth}
g_W(X_i, X_j)=X_i^{\ts}L(\rho)^{\dd} X_j=\begin{cases}
1& \textrm{if $i=j$} \\
0&\textrm{if $i\neq j$}
\end{cases}.
\end{equation}
Under this basis, we compute the following curvature tensors. 
\begin{proposition}[Riemannian, Sectional, Ricci, Scalar curvature]
Given $X_i, X_j, X_k, X_l$. Then the curvature tensor $R^W_{ijkl}=g_W(R^W(X_i,X_j)X_k, X_l)$ satisfies
\begin{equation*}
\begin{split}
R^W_{ijkl}=&\frac{1}{4\lambda_i\lambda_j\lambda_k\lambda_l}\Big\{X_j^{\ts}L(\hat{m}(X_i,X_k))X_l+X_i^{\ts}L(\hat{m}(X_j,X_l))X_k\\
&\hspace{1.5cm}-X_j^{\ts}L(\hat{m}(X_i,X_l))X_k-X_i^{\ts}L(\hat{m}(X_j,X_k))X_l\\
&\hspace{1.5cm}+2 \hat{n}(X_i, X_j)^{\ts}L(\rho)^{\dagger}\hat{n}(X_k,X_l)\\
&\hspace{1.5cm}+\hat{n}(X_i,X_k)^{\ts}L(\rho)^{\dagger}\hat{n}(X_j,X_l)-\hat{n}(X_j,X_k)^{\ts}L(\rho)^{\dagger}\hat{n}(X_i,X_l)\Big\},
\end{split}
\end{equation*}
where \begin{equation*}
\begin{split}
\hat{m}(X_a, X_b):=&-[\lambda_aL(X_a)X_b+\lambda_bL(X_b)X_a ]+\frac{1}{2}L(\rho)(\nabla_GX_a\circ\nabla_GX_b)\\
\hat{n}(X_a, X_b):=&\lambda_aL(X_a)X_b-\lambda_bL(X_b)X_a.
\end{split}
\end{equation*}
Thus 
\begin{itemize}
\item[(i)]The sectional curvature $K^W_{ij}=g_W(R^W(X_i,X_j)X_j, X_i)$ satisfies
\begin{equation*}
\begin{split}
K^{W}_{ij}=& \frac{1}{4\lambda_i^2\lambda_j^2}\Big\{X_j^{\ts} L(\hat{m}(X_i, X_j))X_i + X_i^{\ts}L(\hat{m}(X_j, X_i))X_j\\
&\hspace{0.8cm}-X_j^{\ts}L(\hat{m}(X_i, X_i))X_j- X_i^{\ts}L(m(X_j, X_j))X_i \\
&\hspace{0.8cm}-3 \hat{n}(X_i, X_j)L(\rho)^{\dd}\hat{n}(X_i, X_j)\Big\}.
\end{split}
\end{equation*}
\item[(ii)] The Ricci curvature $\textrm{Ric}^W_{ij}=\sum_{k=1}^{n-1}g_W(R^W(X_i,X_k)X_k, X_j)$ satisfies
\begin{equation*}\label{RicciW}
\begin{split}
\textrm{Ric}^W_{ij}=&\frac{1}{4\lambda_i\lambda_j}\sum_{k=1}^{n-1}\frac{1}{\lambda_k^2}\Big\{X_k^{\ts} L(\hat{m}(X_i, X_k))X_j + X_i^{\ts}L(\hat{m}(X_j, X_k))X_k\\
&\hspace{2cm}-X_k^{\ts} L(\hat{m}(X_i, X_j))X_k- X_i^{\ts} L(\hat{m}(X_k, X_k))X_j\\
&\hspace{2cm}-3\hat{n}(X_i,X_k)^{\ts}L(\rho)^{\dagger}\hat{n}(X_j,X_k)\Big\}.
\end{split}
\end{equation*}
\item[(iii)]The scalar curvature $S^W=\sum_{i=1}^m\sum_{j=1}^mg_W(R^W(X_i,X_j)X_j, X_i)$ satisfies
\begin{equation*}\label{RicciW}
\begin{split}
S^W=&\sum_{i=1}^n\sum_{j=1}^n\frac{1}{4\lambda_i\lambda_j}\sum_{k=1}^{n-1}\frac{1}{\lambda_k^2}\Big\{X_k^{\ts} L(\hat{m}(X_i, X_k))X_j + X_i^{\ts}L(\hat{m}(X_j, X_k))X_k\\
&\hspace{3.2cm}-X_k^{\ts} L(\hat{m}(X_i, X_j))X_k- X_i^{\ts} L(\hat{m}(X_k, X_k))X_j\\
&\hspace{3.2cm}-3\hat{n}(X_i,X_k)^{\ts}L(\rho)^{\dagger}\hat{n}(X_j,X_k)\Big\}.
\end{split}
\end{equation*}
\end{itemize}
\end{proposition}

\subsection{Jacobi equation}
In this sequel, we establish two formulations of Jacobi fields in $(\mathcal{P}_+(G), g_W)$. As in Riemannian geometry, the Jacobi equation refers $\nabla^W_{\dot\rho}\nabla^W_{\dot\rho}J+R^W(J,\dot\rho)\dot\rho=0$, with given initial condition $J(0)$, $\dot J(0)\in\mathbb{R}^n$. I.e. denote $\rho(t)$ as the geodesic and $(X_1(t),\cdots, X_{n-1}(t))$ as an orthonormal frame along $\rho(t)$. Consider $J(t)=\sum_{i=1}^{n-1}a_i(t)X_i(t)$, thus
\begin{equation*}
\ddot a_i+\sum_{k=1}^{n-1}a_kg_W(R^W(X_i,\dot\rho)\dot\rho, X_k)=0, \quad\textrm{for $i=1,2,\cdots, n-1$}.
\end{equation*}
In fact, there is the other interesting formulation of the Jacobi equation by the calculus of variation. Consider an infinitesimal deformation $\rho_\epsilon(t)=\rho(t)+\epsilon h(t)\in\mathcal{P}_+(G)$ with 
$h(t)=(h_i(t))_{i=1}^n\in \big(C^{\infty}[0,1]\big)^n$, $\sum_{i=1}^nh_i(t)=0$, and $h(0)=h(1)=0$, 
\begin{equation}\label{f_svariation}
\mathcal{E}(\rho_\epsilon)=\int_{0}^1\frac{1}{2}\dot\rho_\epsilon(t)^{\ts}L(\rho_\epsilon(t))^{\dd}\dot\rho_\epsilon(t) dt=\mathcal{E}(\rho)+\epsilon\delta \mathcal{E}(\rho)(h)+\frac{\epsilon^2}{2}\delta^2\mathcal{E}(\rho)(h)+o(\epsilon^2).
\end{equation}
\begin{lemma}[Variation of energy]\label{lemma} 
The first and second variation defined in \eqref{f_svariation} satisfies 
\begin{equation*}
\delta \mathcal{E}(\rho)(h)=\int_0^1~ \dot\rho^{\ts}L(\rho)^{\dd}\big(\dot h-\frac{1}{2}L(h)L(\rho)^{\dd}\dot\rho\big)~dt ,
\end{equation*}
and 
\begin{equation}\label{sv}
\delta^2 \mathcal{E}(\rho)(h)=\int_0^1~\big(\dot h-L(h)L(\rho)^{\dd}\dot \rho\big)^{\ts} L(\rho)^{\dd}\big(\dot h-L(h)L(\rho)^{\dd}\dot \rho\big)~dt. 
\end{equation}
\end{lemma}
\begin{proof}
Since $\dot\rho_\epsilon\in T_\rho\mathcal{P}_+(G)$ and Theorem \ref{embed} holds, then $\dot\rho_\epsilon^{\ts}L(\rho_\epsilon)^{\dd}\dot\rho_\epsilon=\dot\rho_\epsilon^{\ts}g(\rho_\epsilon)\dot\rho_\epsilon$.
We then derive the result by the Taylor expansion. 
\begin{equation}\label{variation}
\begin{split}
\mathcal{E}(\rho_\epsilon)=&\int_0^1~ \frac{1}{2}(\dot\rho+\epsilon \dot h)^{\ts} L(\rho+\epsilon h)^{\dd}(\dot\rho+\epsilon\dot h)~ dt=\int_0^1~ \frac{1}{2}(\dot\rho+\epsilon \dot h)^{\ts} g(\rho+\epsilon h)(\dot\rho+\epsilon\dot h)~ dt\\
=&\int_0^1~\frac{1}{2}\dot\rho^{\ts} g(\rho+\epsilon h)\dot \rho+\epsilon \dot h^{\ts}g(\rho+\epsilon h)\dot \rho+\frac{\epsilon^2}{2} \dot h^{\ts}g(\rho+\epsilon h)\dot h~dt  \\
=&\int_0^1~\Big\{\frac{1}{2}\dot\rho^{\ts}L(\rho)^{\dd}\dot \rho+\epsilon \big[\frac{1}{2}\dot\rho^{\ts}\frac{d}{d\epsilon}|_{\epsilon=0}g(\rho+\epsilon h)\dot\rho+\dot h^{\ts} L(\rho)^{\dd}\dot\rho\big]\\
&+\frac{\epsilon^2}{2} \big[\frac{1}{2}\dot\rho^{\ts}\frac{d^2}{d\epsilon^2}|_{\epsilon=0}g(\rho+\epsilon h)\dot \rho+~2\dot h^{\ts}\frac{d}{d\epsilon}|_{\epsilon=0}g(\rho+\epsilon h)\dot \rho ~+~\dot h^{\ts}L(\rho)^{\dd}h\big]\Big\}~dt+o(\epsilon^2).
\end{split}
\end{equation}
To continue the derivation, we need the following claim.

\textbf{Claim 3:}
\begin{equation}\label{claim3}
\left\{
\begin{split}
\frac{d}{d\epsilon}g(\rho+\epsilon h)=&-g(\rho+\epsilon h)\cdot L(h)\cdot g(\rho+\epsilon h)\\
\frac{d^2}{d\epsilon^2}g(\rho+\epsilon h)=&2g(\rho+\epsilon h)\cdot L(h)\cdot g(\rho+\epsilon h)\cdot L(h)\cdot g(\rho+\epsilon h).
\end{split}\right.
\end{equation}
\begin{proof}[Proof of Claim 3]
We show \eqref{claim3} by the following steps. From \eqref{claim1}, we have
\begin{equation*}
\begin{split}
\frac{d}{d\epsilon}g(\rho+\epsilon h)=&-g(\rho+\epsilon h)\cdot \frac{d}{d\epsilon}g(\rho+\epsilon h)^{-1}\cdot g(\rho+\epsilon h)\\
=&-g(\rho+\epsilon h)\cdot \frac{d}{d\epsilon}(L(\rho+\epsilon h)+u_0u_0^{\ts})\cdot g(\rho+\epsilon h)\\
=&-g(\rho+\epsilon h)\cdot L(h)\cdot g(\rho+\epsilon h).
\end{split}
\end{equation*}
Similarly, we have 
\begin{equation*} 
\begin{split}
\frac{d^2}{d\epsilon^2}g(\rho+\epsilon h)=&\frac{d}{d\epsilon}\Big(\frac{d}{d\epsilon}g(\rho+\epsilon h)\Big)=-\frac{d}{d\epsilon}\Big(g(\rho+\epsilon h)\cdot L(h)\cdot g(\rho+\epsilon h)\Big)\\
=&2g(\rho+\epsilon h)\cdot L(h)\cdot g(\rho+\epsilon h)\cdot L(h)\cdot g(\rho+\epsilon h).
\end{split}
\end{equation*}
\end{proof}
 Substituting \eqref{claim3} into \eqref{variation} and using the fact $g(\rho)\dot\rho=L(\rho)^{\dd}\dot\rho$ since $\dot\rho\in T_\rho\mathcal{P}_+(G)$, we derive the first/second variation in Lemma \ref{lemma}. 
 \end{proof}
\begin{theorem}[Jacobi equation]
Consider the geodesic $\rho(t)\in \mathcal{P}_+(G)$ connecting $\rho^0$ and $\rho^1$. The Jacobi equation of $h(t)$ along the geodesic $\rho(t)$ satisfies 
\begin{equation}\label{Jacobi2}
\dot h-L(h)L(\rho)^{\dagger}\dot\rho=0,\quad h(0)=h(1)=0.
\end{equation}
\end{theorem}
\begin{proof}
Denote $\delta^2 \mathcal{E}(\rho)(h)=\int_0^1A(h,\dot h)dt$, where
\begin{equation*}
A(h,\dot h)=\big(\dot h-L(h)L(\rho)^{\dd}\dot \rho\big)^{\ts} L(\rho)^{\dd}\big(\dot h-L(h)L(\rho)^{\dd}\dot \rho\big).
\end{equation*}
The Euler-Lagrangian equation of \eqref{sv}, $\frac{d}{dt}\nabla_{\dot h}A(h,\dot h)=\nabla_hA(h,\dot h)$, satisfies
\begin{equation}\label{EL}
\frac{d}{dt}\Big(L(\rho)^{\dd}\big(\dot h-L(h)L(\rho)^{\dd}\dot\rho\big)\Big)=\Big(\nabla_GL(\rho)^{\dd}\dot\rho\circ\nabla_GL(\rho)^{\dd}\big(\dot h-L(h)L(\rho)^{\dd}\dot\rho\big)\Big).
\end{equation}
On one hand, if $h(t)$ satisfies \eqref{Jacobi2}, then it is a solution of \eqref{EL}. On the other hand, if $h(t)$ satisfies \eqref{EL}, then $\delta^2\mathcal{E}(h)=\frac{1}{2}\int_0^1h^T\Big(\frac{d}{dt}\nabla_{\dot h}A(h,\dot h)-\nabla_hA(h,\dot h)\Big)dt=0$. Since $\delta^2\mathcal{E}(h)=\int_0^1~\big(\dot h-L(h)L(\rho)^{\dd}\dot \rho\big)^{\ts} L(\rho)^{\dd}\big(\dot h-L(h)L(\rho)^{\dd}\dot \rho\big)~dt=0$ and $L(\rho)^{\dd}$ is postive definite in $T_\rho\mathcal{P}_+(G)$, then $h(t)$ satisfies \eqref{Jacobi2}. 
\end{proof}
\subsection{Volume form, Divergence, and Laplace-Beltrami operator}
In this sequel, we study the volume form in $(\mathcal{P}_+(G), g_W)$, based on which we introduce the divergence and Laplace-Beltrami operator.
\begin{theorem}[Volume form]
Denote the volume form of $\mathcal{P}_+(G)$ in the Euclidean metric as $\textrm{vol}$. 
Then the volume form of $(\mathcal{P}_+(G), g_W)$ satisfies 
\begin{equation*}
d\textrm{vol}_{W}={\Pi(\rho)}^{-\frac{1}{2}}d\textrm{vol},\quad\textrm{with}\quad \Pi(\rho):=\Pi_{i=2}^n\lambda_i(\rho),
\end{equation*}
where $\lambda_i(\rho)$ are positive eigenvalues of $L(\rho)$.
\end{theorem}
\begin{proof}
Since $\mathcal{M}_+(G)$ is a smooth oriented manifold with dimension $n$, then $\mathcal{P}_+(G)$ is a submanifold of $\mathcal{M}_+(G)$ with co-dimension $1$.
The orientation of $\mathcal{P}_+(G)$ is induced by its normal
unit vector field $u_0$ and the orientation on $\mathcal{M}_+(G)$. Then the volume form on $(\mathcal{P}_+(G), g_W)$ can be given as follows. Denote the volume form of $(\mathcal{M}_+(G), g_{\mathcal{M}})$ by $\textrm{vol}_g$, then
for any $\rho\in \mathcal{P}_+(G)$ and any $u_1,\cdots, u_{n-1}\in T_\rho\mathcal{P}_+(G)$, 
\begin{equation*}
(d\textrm{vol}_{W})(u_1,\cdots, u_{n-1}) = (d\textrm{vol}_{g})(u_0, u_1,\cdots, u_{n-1}).
\end{equation*}
Since $d\textrm{vol}_g=\sqrt{\textrm{det}(g(\rho))}d\textrm{vol}$ and $g(\rho)u_0=1$, it is clear that $d\textrm{vol}_{W}=\sqrt{\textrm{det}(g(\rho))}d\textrm{vol}={\Pi(\rho)}^{-\frac{1}{2}}d\textrm{vol}$.
\end{proof}
We continue the derivations based on volume forms in $(\mathcal{P}_+(G), g_W)$.
\begin{proposition}[Divergence, Laplace-Beltrami operators]\label{p8}
Denote $\mathcal{G}(\rho)=(\mathcal{G}_i(\rho))_{i=1}^n$, $\mathcal{G}_i(\rho)\in C^\infty (\mathcal{P}_+(G))$. The divergence operator $\textrm{div}_W(\cdot)\colon (C^{\infty}(\mathcal{P}_+(G)))^n\rightarrow C^{\infty}(\mathcal{P}_+(G))$ satisfies
\begin{equation*}
\begin{split}
\textrm{div}_{{W}}\mathcal{G}(\rho)=&\Pi(\rho)^{\frac{1}{2}}\nabla_\rho\cdot(\mathcal{G}(\rho)\Pi(\rho)^{-\frac{1}{2}})=\nabla_\rho\cdot\mathcal{G}(\rho)-\frac{1}{2}\mathcal{G}(\rho)^{\ts}d_\rho\log\Pi(\rho),
\end{split}
\end{equation*}
where $\nabla_\rho\cdot =\sum_{i=1}^n\frac{\partial}{\partial \rho_i}$.

Denote $\mathcal{F}(\rho)\in C^{\infty}(\mathcal{P}_+(G))$. The Laplace-Beltrami operator $\Delta_W\colon C^{\infty}(\mathcal{P}_+(G))\rightarrow C^{\infty}(\mathcal{P}_+(G))$ satisfies 
\begin{equation*}
\begin{split}
\Delta_W \mathcal{F}(\rho)=&\Pi(\rho)^{\frac{1}{2}}\nabla_\rho\cdot(\Pi(\rho)^{-\frac{1}{2}} L(\rho)d_\rho\mathcal{F}(\rho))\\
=&\textrm{tr}(L(\rho)\cdot d^2_\rho\mathcal{F}(\rho))-\frac{1}{2}(d_\rho\mathcal{F}(\rho))^{\ts}L(\rho)\big(d_\rho\log\Pi(\rho)\big)\\
=&\sum_{(i,j)\in \vec{E}}\omega_{ij}\theta_{ij}(\rho)(\frac{\partial^2}{\partial\rho_i^2}-2\frac{\partial^2}{\partial\rho_i\partial\rho_j}+\frac{\partial^2}{\partial\rho_j^2})\mathcal{F}(\rho)-\frac{1}{2}\sum_{(i,j)\in \vec{E}}\nabla_{ij}d_\rho\log\Pi(\rho)\cdot\nabla_{ij}d_\rho\mathcal{F}(\rho)\theta_{ij}(\rho).
\end{split}
\end{equation*}
\end{proposition}
\begin{proof}
Consider a test function $\mathcal{F}(\rho)\in C^{\infty}(\mathcal{P}_+(G))$ with compact support in $\mathcal{P}_+(G)$. Then 
\begin{equation*}
\begin{split}
\int_{\mathcal{P}_+(G)}g_W(\nabla_W\mathcal{F}(\rho), \mathcal{G}(\rho))d\textrm{vol}_W=&\int_{\mathcal{P}_+(G)}d_\rho\mathcal{F}(\rho)\cdot \mathcal{G}(\rho)\Pi(\rho)^{-\frac{1}{2}}d\textrm{vol}\\
=&-\int_{\mathcal{P}_+(G)}\mathcal{F}(\rho)\nabla_\rho\cdot(\mathcal{G}(\rho)\Pi(\rho)^{-\frac{1}{2}})d\textrm{vol}\\
=&-\int_{\mathcal{P}_+(G)}\mathcal{F}(\rho)\Pi(\rho)^{\frac{1}{2}}\nabla_\rho\cdot(\mathcal{G}(\rho)\Pi(\rho)^{-\frac{1}{2}}) d\textrm{vol}_W,
\end{split}
\end{equation*}
which finishes the proof.

From the divergence operator and noticing $\nabla_W\mathcal{F}(\rho)=L(\rho)d_\rho\mathcal{F}(\rho)$, we have
\begin{equation*}
\Delta_W\mathcal{F}(\rho)=\textrm{div}_W(\nabla_W\mathcal{F}(\rho))=\nabla_{\rho}\cdot(L(\rho)d_\rho\mathcal{F}(\rho))+(d_\rho\mathcal{F}(\rho))^{\ts}L(\rho)(d_\rho\log\Pi(\rho)^{-\frac{1}{2}}).
\end{equation*}
Since 
\begin{equation*}
\begin{split}
&\nabla_{\rho}\cdot(L(\rho)d_\rho\mathcal{F}(\rho))\\=&\sum_{i=1}^n\frac{\partial}{\partial \rho_i}\big(\sum_{j\in N(i)}(\frac{\partial}{\partial\rho_i}-\frac{\partial}{\partial\rho_j})\mathcal{F}(\rho)\theta_{ij}(\rho)\omega_{ij}\big)\\
=&\sum_{i=1}^n\sum_{j\in N(i)}(\frac{\partial^2}{\partial\rho^2_i}-\frac{\partial^2}{\partial\rho_j\partial\rho_i})\mathcal{F}(\rho)\theta_{ij}(\rho)\omega_{ij}-\frac{1}{2}\sum_{i=1}\sum_{j\in N(i)}(\frac{\partial}{\partial\rho_i}-\frac{\partial}{\partial\rho_j})\mathcal{F}(\rho)\omega_{ij}\\
=&\sum_{(i,j)\in \vec{E}}(\frac{\partial^2}{\partial\rho^2_i}-\frac{\partial^2}{\partial\rho_j\partial\rho_i})\mathcal{F}(\rho)\theta_{ij}(\rho)\omega_{ij}-\frac{1}{2}\sum_{(i,j)\in \vec{E}}(\frac{\partial}{\partial\rho_i}-\frac{\partial}{\partial\rho_j})\mathcal{F}(\rho)\omega_{ij}\\
+&\sum_{(j,i)\in \vec{E}}(\frac{\partial^2}{\partial\rho^2_i}-\frac{\partial^2}{\partial\rho_j\partial\rho_i})\mathcal{F}(\rho)\theta_{ij}(\rho)\omega_{ij}-\frac{1}{2}\sum_{(j,i)\in \vec{E}}(\frac{\partial}{\partial\rho_i}-\frac{\partial}{\partial\rho_j})\mathcal{F}(\rho)\omega_{ij}\qquad \textrm{Relabel $i$ by $j$.}\\
=&\sum_{(i,j)\in \vec{E}}(\frac{\partial^2}{\partial\rho^2_i}-2\frac{\partial^2}{\partial\rho_j\partial\rho_i}+\frac{\partial^2}{\partial\rho_j^2})\mathcal{F}(\rho)\theta_{ij}(\rho)\omega_{ij}\\
=&\textrm{tr}(L(\rho)\cdot d^2_\rho\mathcal{F}(\rho)),
\end{split}
\end{equation*} 
and $(d_\rho\mathcal{F}(\rho))^{\ts}L(\rho)(d_\rho\log\Pi(\rho)^{-\frac{1}{2}})=\sum_{(i,j)\in \vec{E}}\nabla_{ij}d_\rho\log\Pi(\rho)^{-\frac{1}{2}}\cdot\nabla_{ij}d_\rho\mathcal{F}(\rho)\theta_{ij}(\rho)$, we prove the result.
\end{proof}

In the last, we find the Hessian operator on $(\mathcal{P}_+(G), g_W)$, i.e. $\textrm{Hess}_W(\cdot, \cdot)\colon C^{\infty}(T\mathcal{P}_+(G))\times C^{\infty}(T\mathcal{P}_+(G))\rightarrow C^{\infty}(\mathcal{P}_+(G))$.
\begin{proposition}[Hessian operator]
Given $\sigma_1,\sigma_2\in T_\rho\mathcal{P}_+(G)$, then
\begin{equation}\label{hess}
\begin{split}
\textrm{Hess}_{W}\mathcal{F}(\rho)(\sigma_1, \sigma_2)=&\sigma_1^{\ts}d^2_\rho\mathcal{F}(\rho)\sigma_2+ \frac{1}{2}\Big\{d_\rho\mathcal{F}(\rho)^{\ts}L(\sigma_1)L(\rho)^{\dd}\sigma_2+d_\rho\mathcal{F}(\rho)^{\ts}L(\sigma_2)L(\rho)^{\dd}\sigma_1\\
&\hspace{2.1cm}\quad- d_\rho\mathcal{F}(\rho)^{\ts}L(\rho)\Big(\nabla_GL(\rho)^{\dd}\sigma_1\circ\nabla_GL(\rho)^{\dd}\sigma_2\Big)\Big\}.
\end{split}
\end{equation}
\end{proposition}
\begin{proof}
From the definition of Hessian on a manifold, we have
\begin{equation}\label{dH}
\begin{split}
\textrm{Hess}_W\mathcal{F}(\rho)(\sigma_1,\sigma_2)=&g_W(\sigma_1, \nabla_{\sigma_2}\nabla_W\mathcal{F}(\rho))\\
=&\sigma_2\big(g_W(\sigma_1,\nabla_W\mathcal{F}(\rho))\big)-g_W(\nabla_{\sigma_2}\sigma_1, \nabla_W\mathcal{F}(\rho)).
\end{split}
\end{equation}
Denote $\frac{d\rho(t)}{dt}|_{t=0}=\sigma_2$. Then 
\begin{equation*}
\sigma_2\big(g_W(\sigma_1,\nabla_W\mathcal{F}(\rho))\big)=\frac{d}{dt}|_{t=0}\big(\sigma_1^{\ts}d_\rho\mathcal{F}(\rho)\big)=\sigma_1^{\ts}d^2_\rho\mathcal{F}(\rho)\sigma_2.
\end{equation*}
Substituting the above formula and Proposition \ref{Levi} into \eqref{dH}, we finish the proof.
\end{proof}
As in a Riemannian manifold $(\mathcal{P}_+(G), g_W)$, given the orthonormal basis $X_i=\sqrt{\lambda_i(\rho)}u_i(\rho)$, $i=1,\cdots,n-1$, it is clear that
\begin{equation*}
\Delta_W\mathcal{F}(\rho)=\sum_{i=1}^{n-1}\textrm{Hess}_{W}\mathcal{F}(\rho)(X_i, X_i).
\end{equation*}

\section{Riemannian calculus on density manifold}\label{section5}
In this section, the approach in previous sections guides us to derive all geometry formulas, in the sense of \cite{IM}, on the space of probability densities supported on $M$. We present the results for the completeness of this paper. 

Suppose $(M,d_M)$ is a smooth, compact, connected, $d$-dimensional Riemannian manifold without boundary. $d_M$ is the Riemannian metric.
A volume element is a positive $d$-form represented by $dx$. The total volume of manifold is $\textrm{vol}(M)$. The gradient, divergence operator in $M$ is denoted by $\nabla$, $\nabla\cdot$ respectively. $\Delta$ is the Laplacian-Beltermi operator on $M$. 
\subsection{Review of density manifold}
We briefly review the $L^2$-Wassertein metric in continuous states $M$. 
Denote $\mathcal{P}_2(M)$ as the set of probability density functions with finite second moment. Given $\rho^0(x)$ and $\rho^1(x)\in \mathcal{P}_2(M)$, the $L^2$-Wasserstein distance between $\rho^0$ and $\rho^1$ is given by
\begin{equation*}
(W(\rho^0, \rho^1))^2=\inf_{\pi\in \Pi(\rho^0, \rho^1)}{\int_{M}\int_Md_M(x,y)^2\pi(dx,dy)},
\end{equation*}
where $\Pi$ are joint measures supported on $M\times M$ with marginals $\rho^0$ and $\rho^1$. Equivalently, the $L^2$-Wasserstein metric can be written as the variational problem of the probability average of kinetic energy. Denote the path $\rho_t=\rho(t,x)$ connecting $\rho^0(x)$ and $\rho^1(x)\in \mathcal{P}_+(M)$, then
\begin{equation}\label{BB1}
(W(\rho^0,\rho^1))^2=\inf_{v_t,\rho_t}~~\big\{{\int_0^1\int_{M} v_t^2\rho_t dx dt} \colon \frac{\partial \rho_t}{\partial t}+\nabla\cdot (\rho_tv_t)=0,~\rho_0=\rho^0,~ \rho_1=\rho^1\big\},
 \end{equation}
where the infimum is taken among all Borel vector field functions $v_t=v(t,x)\in T_xM$ and density path $\rho_t$. Here the equivalence between \eqref{BB1} and the above linear programming problem can be shown by the duality argument and the Hopf-Lax formula in $(M, d_M)$; see details in \cite{vil2008}. 

From now on, we focus on \eqref{BB1}. Notice that $\mathcal{P}_2(M)$ is an infinite dimensional manifold with boundary. Here the boundary refers the set in which the density function is zero at certain point. For better illustration, consider the space of positive smooth density functions supported on $M$,
\begin{equation*}
\mathcal{P}_+(M)=\{\rho(x)\in C^{\infty}(M)\colon \rho(x)>0,~\int_M\rho(x)dx=1\}\subset \mathcal{P}_2(M).
\end{equation*}
In literature \cite{Lafferty}, $\mathcal{P}_+(M)$ is called density manifold. The following linear operator is needed for describing the geometry of density manifold. 
\begin{definition}
Given $a(x)\in C^{\infty}(M)$, define the weighted Laplacian operator $\Delta_a\colon C^{\infty}(M)\rightarrow C^{\infty}(M)$,
\begin{equation*}
\Delta_a\Phi(x)=\nabla\cdot(a(x)\nabla\Phi(x)),\quad \Phi(x)\in C^{\infty}(M).
\end{equation*}
\end{definition}
Let $a(x)=\rho(x)\in \mathcal{P}_+(M)$, the elliptic operator of $\Delta_\rho$ can be used to show the Hodge decomposition on $M$. For any smooth vector field $v(x)\in T_xM$, there exists a potential $\Phi(x)\in C^{\infty}(M)$ module constant shrift, and a divergence free vector field $\Psi(x)\in T_{x}M$, such that 
\begin{equation*}
v(x)=\nabla\Phi(x)+\Psi(x),\quad \nabla\cdot(\rho(x)\Psi(x))=0.
\end{equation*}
In other words, $\int_{M}v(x)^2\rho(x)dx=\int_M[(\nabla\Phi(x))^2+\Psi(x)^2]\rho(x)dx\geq \int_M(\nabla\Phi(x))^2\rho(x)dx$. Thus the metric $W$ defined in \eqref{metric_BB} is equivalent to 
\begin{equation*} 
\begin{split}
\left(W( \rho^0, \rho^1)\right)^2=&\inf_{\Phi_t\colon\rho_0= \rho^0,~\rho_1= \rho^1}\{\int_0^1\int_M(\nabla\Phi_t)^2\rho_tdx dt ~:~\partial_t\rho_t+\nabla\cdot( \rho_t \nabla\Phi_t)=0\}\\
=&\inf_{\Phi_t\colon \rho_0= \rho^0,~\rho_1= \rho^1}\{\int_0^1\int_M\Phi_t(-\Delta_{\rho_t}\Phi_t)dx dt ~:~\partial_t\rho_t=-\Delta_{\rho_t}\Phi_t\}.
\end{split}
\end{equation*}

Denote the tangent space at $\rho \in \mathcal{P}_+(M)$ by 
\begin{equation*}
T_\rho \mathcal{P}_+({M})=\{\sigma(x)\in C^{\infty}(M)\colon\int_{M}\sigma(x)dx=0\}.
\end{equation*}
Given $\sigma_1,\sigma_2\in T_\rho\mathcal{P}_+(M)$, the inner product $g_{W}(\cdot, \cdot):T_\rho\mathcal{P}_+(M)\times T_\rho\mathcal{P}_+(M)  \rightarrow \mathbb{R}$ is defined by
\begin{equation*}
g_{W}(\sigma_1,\sigma_2):=\int_{M}{\sigma_1}(x) (-\Delta_\rho)^{\dagger}\sigma_2(x)dx,
\end{equation*}
where $(-\Delta_\rho)^{\dd}\colon T_\rho\mathcal{P}_+(M)\rightarrow  T_\rho\mathcal{P}_+(M)$ is the pseudo inverse operator of $-\nabla\cdot(\rho(x)\nabla)$. 

If we denote $\Phi_i(x)\in C^{\infty}(M)$ modulo additive constants, such that $(-\Delta_\rho\Phi_i)(x)=\sigma_i(x)$, $i=1,2$. 
Then \begin{equation*}
\begin{split}
g_{W}(\sigma_1,\sigma_2)=&\int_{M}  \Phi_1(x)(-\Delta_\rho)(-\Delta_\rho)^{\dagger}(-\Delta_\rho)\Phi_2(x)dx\\
=&\int_{M}  \Phi_1(x)(-\nabla\cdot(\rho(x)\nabla\Phi_2(x)))dx\\
=&\int_{M}  \nabla\Phi_1(x)\nabla\Phi_2(x)\rho(x)dx.
\end{split}
\end{equation*}

Denote the path $\rho_t=\rho(t,x)$ connecting $\rho^0(x)$ and $\rho^1(x)\in \mathcal{P}_+(M)$, and let $\partial_t\rho_t=-\Delta_\rho\Phi_t=-\nabla\cdot(\rho_t\nabla\Phi_t)$. Then the metric \eqref{BB1} can be represented by
\begin{equation*}
\begin{split}
(W(\rho^0,\rho^1))^2=&\inf_{\rho_t\in\mathcal{P}_+(M)}\big\{\int_0^1g_W(\partial_t\rho_t, \partial_t\rho_t)dt\colon \rho_0=\rho^0,\rho_1=\rho^1\big\}.
\end{split}
\end{equation*}
By the arc-length time reparameterization, one can simply denote the $L^2$-Wasserstein metric by 
\begin{equation*}
\begin{split}
W(\rho^0,\rho^1)=&\inf_{\rho_t\in\mathcal{P}_+(M)}\big\{\int_0^1\sqrt{g_W(\partial_t\rho_t, \partial_t\rho_t)}dt\colon \rho_0=\rho^0,\rho_1=\rho^1\big\}.
\end{split}
\end{equation*}

\subsection{Riemannian calculus}
Similar as in section \ref{external}, we first construct a Riemannian metric in the positive measure space, and embed the density manifold as its submanifold. The geometry structures of density manifold follows directly from the ones in the positive measure space.

Consider
\begin{equation*}
\mathcal{M}_+(M)=\{\mu(x)\in C^{\infty}(M)\colon\mu(x)>0\}.
\end{equation*}
Thus $\mathcal{P}_+(M)\subset \mathcal{M}_+(M)$. Denote the tangent space at $\mu\in \mathcal{M}_+(M)$,  
\begin{equation*}
T_\mu\mathcal{M}_+(M)=\{A(x)\in C^{\infty}(M)\},
\end{equation*}
and 
\begin{equation*}
T\mathcal{M}_+(M)=\{(\mu(x),A(x))\colon \mu(x)\in \mathcal{M}_+(M),~A(x)\in C^{\infty}(M)\}.
\end{equation*}
We define a Riemannian inner product on the infinite dimensional manifold $\mathcal{M}_{+}(M)$.

Given $\mu\in \mathcal{M}_+(M)$ and $A(x)\in C^{\infty}(M)$, denote a positive definite operator $g(\mu)\colon C^\infty(M)\rightarrow C^\infty(M)$ by \begin{equation}\label{gmu}
(g(\mu)A)(x)=\big(-\Delta_\mu\big)^{\dd}(A(x)-\int_{M}A(y)dy)+\int_{M}A(y)dy.
\end{equation}
\begin{definition}[Inner product in $\mathcal{M}_+(M)$]
Define the inner product $g_{\mathcal{M}}\colon C^{\infty}(T\mathcal{M}_+(M))\times C^{\infty}(T\mathcal{M}_+(M))\rightarrow C^{\infty}(\mathcal{M}_+(M))$ by
\begin{equation*}
\begin{split}
g_{\mathcal{M}}(A_1(x), A_2(x))=&\int_M A_1(x) (g(\mu)A_2)(x)dx\\
=&\int_{M}(A_1(x)-\int_{M}A(y)dy)(-\Delta_\mu)^{\dd}(A_2(x)-\int_MA_2(y)dy)dx\\
&+\int_{M}A_1(y)dy\int_{M}A_2(y)dy,
\end{split}
\end{equation*}
for any $A_1(x)$, $A_2(x)\in T_\mu\mathcal{M}_+(M)$.
\end{definition}
It is clear that $(\mathcal{M}_+(M), g_{\mathcal{M}})$ is an infinite dimensional Riemannian manifold. We next show that $(\mathcal{M}_+(M), g_{\mathcal{M}})$ induces a metric in its submanifold $(\mathcal{P}_+(M), W)$. Denote $\imath\colon\mathcal{P}_+(M)\rightarrow \mathcal{M}_+(M)$ a natural inclusion by $\imath(\rho)=\rho$, then $\imath$ induces a Riemannian metric $W$ on $\mathcal{P}_+(M)$ via pullback, i.e. for any $\sigma_1(x),\sigma_2(x)\in T_\rho\mathcal{P}_+(M)$, \begin{equation*}\label{relation_C}
g_W(\sigma_1,\sigma_2)=g_{\mathcal{M}}(\sigma_1,\sigma_2)=\int_{M}\sigma_1(x)(-\Delta_\rho)^{\dd}\sigma_2(x)dx.
\end{equation*}

We are ready to find the Riemannian structure of $(\mathcal{P}_+(M), W)$ by the one in $(\mathcal{M}_+(M), g_{\mathcal{M}})$. %For example, we derive the gradient vector in $(\mathcal{P}_+(M), W)$.
\begin{proposition}[Gradient]
Consider $\mathcal{F}(\rho)\in C^{\infty}(\mathcal{P}_+(M))$, denote its gradient operators in $(\mathcal{M}_{+}(M),g_{\mathcal{M}})$ and $(\mathcal{P}_+(M), W)$ by 
$\nabla_g\mathcal{F}(\rho)\in T_{\rho}\mathcal{M}_+(M)$, $\nabla_W\mathcal{F}(\rho)\in T_{\rho}\mathcal{P}_+(M)$, respectively. Then \begin{equation*}
\nabla_W\mathcal{F}(\rho)=\nabla_g\mathcal{F}(\rho)-g(\nabla_g\mathcal{F}(\rho), u_0)u_0=-\nabla\cdot(\rho(x)\nabla \frac{\delta}{\delta\rho(x)}\mathcal{F}(\rho)),
\end{equation*}
where $\frac{\delta}{\delta\rho(x)}$ represents the $L^2$ first variation and $u_0(x)\equiv1$, is the orthogonal projection related to $g$ on $\mathcal{M}_+(M)$, of the restriction of $\nabla_g\mathcal{F}(\rho)$ to $\mathcal{P}_+(M)$.
\end{proposition}
\begin{proof}
It is clear that 
\begin{equation*}
\nabla_g\mathcal{F}(\rho)(x)=(g^{-1}(\rho)\delta_{\rho}\mathcal{F}(\rho))(x)
\end{equation*}
We show that $g^{-1}(\rho)\colon C^{\infty}(M)\rightarrow C^{\infty}(M)$ defined by
\begin{equation*}
(g^{-1}(\rho)A)(x)=-\Delta_\rho A(x)+\int_{M}A(y)dy,\quad \textrm{for any $A(x)\in C^{\infty}(M)$.}
\end{equation*}
We show that $g^{-1}(\rho)$ is the inverse of operator $g(\rho)$ by 
\begin{equation*}
\begin{split}
g^{-1}(\rho)(g(\rho)A)(x)=&(-\Delta_\rho)(-\Delta_\rho)^{\dd}\big(A(x)-\int_MA(y)dy\big)+\int_MA(y)dy\\
=&A(x)-\int_MA(y)dy+\int_MA(y)dy=A(x).
\end{split}
\end{equation*}
Thus 
\begin{equation*}
\nabla_g\mathcal{F}(\rho)=g(\rho)^{-1}\delta_{\rho}\mathcal{F}(\rho)=-\Delta_\rho \delta_{\rho}\mathcal{F}(\rho)+\int_M \frac{\delta}{\delta\rho(y)}\mathcal{F}(\rho)dy .
\end{equation*}
Denote $u_0(x)=1$. For any $\sigma\in T_\rho\mathcal{P}_+(M)$, then
\begin{equation*}
\begin{split}
\int_M\sigma(x) (g(\rho)u_0)(x)dx=&\int_M\sigma(x)\big(-\Delta_\rho^{\dd}u_0(x)+\int_{M}u_0dy\big)dx \\
=&\textrm{vol}(M)\int_M\sigma(x)dx=0.
\end{split}
\end{equation*}
Thus the unit norm vector at $\rho\in\mathcal{P}_+(G)$ is a constant function $u_0(x)=1$. 
Notice 
\begin{equation*}
g(\nabla_g\mathcal{F}(\rho), u_0)=\int_M(-\Delta_\rho \delta_{\rho}\mathcal{F}(\rho)+\int_M \frac{\delta}{\delta\rho(x)}\mathcal{F}(\rho)dy)dy=\int_M \frac{\delta}{\delta\rho(y)}\mathcal{F}(\rho)dy,
\end{equation*}
then the orthogonal projection of $\nabla_g\mathcal{F}(\rho)$ is
\begin{equation*}
\begin{split}
\nabla_g\mathcal{F}(\rho)-g(\nabla_g\mathcal{F}(\rho), u_0)u_0=&-\Delta_\rho \delta_{\rho}\mathcal{F}(\rho)+\int_M \frac{\delta}{\delta\rho(y)}\mathcal{F}(\rho)dy-\int_M \frac{\delta}{\delta\rho(y)}\mathcal{F}(\rho)dy\\
=&-\Delta_\rho\delta_{\rho}\mathcal{F}(\rho)=-\nabla\cdot(\rho(x)\nabla \frac{\delta}{\delta\rho(x)}\mathcal{F}(\rho)),
\end{split}
\end{equation*}
which finishes the proof.
\end{proof}
We derive the commutator on density manifold, i.e. $[\cdot,\cdot]_W\colon C^{\infty}(T\mathcal{P}_+(M))\times C^{\infty}(T\mathcal{P}_+(M))\rightarrow C^{\infty}(\mathcal{P}_+(M))$.
\begin{proposition}[Commutator]
Given $\sigma_1(x)$, $\sigma_2(x)\in T_\rho\mathcal{P}_+(M)$, then 
\begin{equation*}
[\sigma_1,\sigma_2]_W=0.
\end{equation*}
\end{proposition}
\begin{proof}
Consider $\mathcal{F}\in C^\infty(\mathcal{P}_+(M))$. Since $\sigma_1(\mathcal{F}(\rho)):=\frac{d}{dt}|_{t=0}\mathcal{F}(\rho+t\sigma_1)$, then 
\begin{equation*}
\begin{split}
[\sigma_1,\sigma_2]_W(\mathcal{F}(\rho))=&\frac{d}{dt}|_{t=0}\frac{d}{ds}|_{s=0}\mathcal{F}(\rho+t\sigma_1+s\sigma_2)-\frac{d}{ds}|_{s=0}\frac{d}{dt}|_{t=0}\mathcal{F}(\rho+t\sigma_1+s\sigma_2)\\
=&\int_{M}\int_M\frac{\delta^2}{\delta\rho(x)\delta\rho(y)}\mathcal{F}(\rho)\sigma_1(x)\sigma_2(y)-\frac{\delta^2}{\delta\rho(x)\delta\rho(y)}\mathcal{F}(\rho)\sigma_1(y)\sigma_2(x) dxdy=0,
\end{split}
\end{equation*}
where $\frac{\delta^2}{\delta\rho(x)\delta\rho(y)}$ is the $L^2$ second variation of $\mathcal{F}(\rho)$.
\end{proof}
Following the commutator, we derive the Levi-Civita connection. 
\begin{proposition}[Levi-Civita connection and Christoffel symbol]\label{Levi}
The operator $$\nabla^{W}_{\cdot}\cdot\colon C^{\infty}(T\mathcal{P}_+(M))\times C^{\infty}(T\mathcal{P}_+(M))\rightarrow C^{\infty}(T\mathcal{P}_+(M))$$ 
is defined by
\begin{equation*}
\begin{split}
\nabla^W_{\sigma_1(x)}\sigma_2(x)=&-\frac{1}{2}\big(\Delta_{\sigma_1}\Delta_\rho^{\dagger}\sigma_2+\Delta_{\sigma_2}\Delta_\rho^{\dagger}\sigma_1+\Delta_\rho(\nabla \Delta_\rho^{\dagger}\sigma_1\cdot\nabla \Delta_\rho^{\dagger}\sigma_2)\big)(x).
\end{split}
\end{equation*}
\end{proposition}
\begin{proof}
We apply the Koszul formula for $\sigma_i\in \mathcal{T}_\rho\mathcal{P}_+(M)$, $i=1,2,3$. Notice $[\sigma_i, \sigma_j]_W=0$, then
\begin{equation}\label{LC}
\begin{split}
&g_W(\nabla_{\sigma_1}^W\sigma_2, \sigma_3)=\frac{1}{2}\Big\{\sigma_1(g_W(\sigma_2,\sigma_3))+\sigma_2(g_W(\sigma_1,\sigma_3))-\sigma_3(g_W(\sigma_1,\sigma_2))\Big\}\\
=&\frac{1}{2}\frac{d}{dt}|_{t=0}\int_M\Big\{\sigma_2(x)(-\Delta_{\rho+t\sigma_1})^{\dagger}\sigma_3(x)+\sigma_1(x)(-\Delta_{\rho+t\sigma_2})^{\dagger}\sigma_3(x)-\sigma_1(x)(-\Delta_{\rho+t\sigma_3})^{\dagger}\sigma_2(x)\Big\}dx.
\end{split}
\end{equation}
%%Since $g_W(\nabla_{\sigma_1}^W\sigma_2, \sigma_3)=(\nabla_{\sigma_1}^W\sigma_2)^{\ts}L(\rho)^{\dagger}\sigma_3$, for any $\sigma_3\in \mathcal{P}_+(G)$, the result is proved.
The following two claims are needed to further derive \eqref{LC}.

\noindent\textbf{Claim 3:}
\begin{equation*}
\frac{d}{dt}|_{t=0}\int_M\sigma_2(x)(-\Delta_{\rho+t\sigma_1}^{\dagger})\sigma_3(x)dx=\int_M\sigma_2(x)\Delta_{\rho}^{\dagger}\Delta_{\sigma_1}\Delta_\rho^{\dagger}\sigma_3(x)dx.
\end{equation*}
\begin{proof}[Proof of Claim 3]
Since 
 \begin{equation*}
0=\frac{d}{dt}(g(\rho+t\sigma_1)g(\rho+t\sigma_1)^{-1})=\frac{d}{dt}g(\rho+t\sigma_1) g(\rho+t\sigma_1)^{-1}+g(\rho+t\sigma_1)\frac{d}{dt}g(\rho+t\sigma_1)^{-1},
\end{equation*}
then \begin{equation}\label{a}
\begin{split}
\frac{d}{dt}g(\rho+t\sigma_1)=&-g(\rho+t\sigma_1)\frac{d}{dt}g(\rho+t\sigma_1)^{\dagger}g(\rho+t\sigma_1)\\
=&-g(\rho+t\sigma_1)\frac{d}{dt}(-\Delta_{\rho+t\sigma_1})g(\rho+t\sigma_1)\\ 
=&g(\rho+t\sigma_1)\Delta_{\sigma_1}g(\rho+t\sigma_1),
\end{split}
\end{equation}
where the second equality is true because the weighted Laplacian operator $\Delta_\rho=\nabla\cdot(\rho(x)\nabla)$ is linear w.r.t. $\rho$. From \eqref{relation}, we have 
\begin{equation}\label{b}
\int_M\sigma_2(x)(-\Delta^{\dagger}_\rho)\sigma_3(x)dx=\int_M\sigma_2(x)\big(g(\rho+t\sigma_1)\sigma_3\big)(x)dx.
\end{equation}
 From \eqref{a} and \eqref{b}, we have \begin{equation*}\label{c}
\begin{split}
\frac{d}{dt}|_{t=0}\int_M\sigma_2(x)(-\Delta_{\rho+t\sigma_1})^{\dagger}\sigma_3(x)dx=&\int_M\sigma_2(x)\frac{d}{dt}|_{t=0}g(\rho+t\sigma_1)\sigma_3(x)dx\\
=&\int_{M} \sigma_2(x)g(\rho) (-\Delta_{\sigma_1}) g(\rho)\sigma_3(x) dx\\
=&\int_{M} \sigma_2(x)\Delta_\rho^{\dd} \Delta_{\sigma_1}\Delta_\rho^{\dd}\sigma_3(x) dx.
\end{split}
\end{equation*}
where the last equality follows from the fact that, for any $\sigma(x)\in T_\rho\mathcal{P}_+(M)$,
\begin{equation*}
g(\rho)\sigma(x)=(-\Delta_\rho)^{\dd}\sigma(x)+\int_{M}\sigma(x)dx=(-\Delta_\rho)^{\dd}\sigma(x).
\end{equation*}
\end{proof}
\noindent\textbf{Claim 4:}
\begin{equation*}
\int_M\sigma_1(x)\Delta_\rho^{\dagger}\Delta_{\sigma_3}\Delta_\rho^{\dagger}\sigma_2(x)dx=-\int_M\sigma_3(x)\Delta_\rho^{\dagger}\Delta_\rho\Big(\nabla\Delta_{\rho}^{\dagger}\sigma_1(x)\cdot \nabla \Delta_\rho^{\dagger}\sigma_2(x)\big)dx.
\end{equation*}
\begin{proof}[Proof of Claim 4]
\begin{equation*}
\begin{split}
\int_M\sigma_1(x)\Delta_\rho^{\dagger}\Delta_{\sigma_3}\Delta_\rho^{\dagger}\sigma_2(x)dx=&\int_M\Delta_\rho^{\dagger}\sigma_1(x)\nabla\cdot(\sigma_3(x) \nabla\Delta_{\rho}^{\dagger}\sigma_2(x))dx\\
=&-\int_M\sigma_3(x)\big(\nabla\Delta_\rho^{\dagger}\sigma_1(x)\cdot\nabla\Delta_{\rho}^{\dagger}\sigma_2(x)\big)dx\\
=&-\int_M\sigma_3(x)\Delta_\rho^{\dagger}\Delta_\rho\Big(\nabla\Delta_{\rho}^{\dagger}\sigma_1(x)\cdot \nabla \Delta_\rho^{\dagger}\sigma_2(x)\big)dx.
\end{split}
\end{equation*}
\end{proof}
Applying Claim 3 and 4, we have
\begin{equation*}
\begin{split}
&g_W(\nabla_{\sigma_1}^W\sigma_2, \sigma_3)\\
=&\frac{1}{2}\int_M\sigma_3(x)\Delta_{\rho}^{\dagger}\Delta_{\sigma_1}\Delta_\rho^{\dagger}\sigma_2(x)+\sigma_3(x)\Delta_{\rho}^{\dagger}\Delta_{\sigma_2}\Delta_\rho^{\dagger}\sigma_1(x)-\sigma_1(x)\Delta_{\rho}^{\dagger}\Delta_{\sigma_3}\Delta_\rho^{\dagger}\sigma_2(x)dx\\
=&\int_M\sigma_3(x)(-\Delta_{\rho})^{\dagger}(-\frac{1}{2})\big\{\Delta_{\sigma_1}\Delta_\rho^{\dagger}\sigma_2(x)+\sigma_3(x)\Delta_{\rho}^{\dagger}\Delta_{\sigma_2}\Delta_\rho^{\dagger}\sigma_1(x)-\sigma_1(x)\Delta_{\rho}^{\dagger}\Delta_{\sigma_3}\Delta_\rho^{\dagger}\sigma_2(x)\big\}dx.
\end{split}
\end{equation*}
From the definition of inner product $g_W$, we finish the proof.
 \end{proof}
 By the Levi-Civita connection, we define the Christoffel symbol in density manifold. 
 \begin{definition}
 Denote the Christoffel symbol operator at $\rho\in \mathcal{P}_+(M)$ by $\Gamma^{W,x}\colon T_\rho\mathcal{P}_+(M)\times  T_\rho\mathcal{P}_+(M)\rightarrow \mathbb{R}$,
 \begin{equation*}
\begin{split}
 \Gamma^{W,x}(\sigma_1,\sigma_2)=-\frac{1}{2}\big(\Delta_{\sigma_1}\Delta_\rho^{\dagger}\sigma_2+\Delta_{\sigma_2}\Delta_\rho^{\dagger}\sigma_1+\Delta_\rho(\nabla \Delta_\rho^{\dagger}\sigma_1\cdot\nabla \Delta_\rho^{\dagger}\sigma_2)\big)(x).
 \end{split}
 \end{equation*}
 \end{definition}
 It is clear $\nabla^W_{\sigma_1}\sigma_2(x)=\nabla^W_{\sigma_2}\sigma_1(x)=\Gamma^{W,x}(\sigma_1,\sigma_2)$. Following the Christoffel symbol, the parallel transport and geodesic equation in density manifold can be derived directly. 
\begin{proposition}[Parallel transport]
Denote $\rho\colon (a,b)\rightarrow \mathcal{P}_+(M)$. Consider $\sigma_t=\sigma(t,x)\in T_\rho\mathcal{P}_+(G)$ be a vector field along curve $\rho_t=\rho(t,x)$, then the equation for $\sigma_t$ to be parallel along $\rho_t$ is
\begin{equation*}
\partial_t\sigma_t=\frac{1}{2}\big(\Delta_{\sigma_t}\Delta_{\rho_t}^{\dagger}\partial_t\rho_t+\Delta_{\partial_t\rho_t}\Delta_{\rho_t}^{\dagger}\sigma_t+\Delta_{\rho_t}(\nabla\Delta_\rho^{\dagger}\partial_t\rho_t\cdot\nabla\Delta_\rho^{\dagger}\sigma_t)\big).
\end{equation*}
Let $\sigma_t=\partial_t\rho_t$, then the geodesic equation satisfies
\begin{equation}\label{geo}
\partial_{tt}\rho_t=\Delta_{\partial_t\rho_t}\Delta_{\rho_t}^{\dagger}\partial_t\rho_t+\frac{1}{2}\Delta_{\rho_t}(\nabla \Delta_{\rho_t}^{\dagger}\partial_t\rho_t)^2.
\end{equation}
\end{proposition}
\begin{proof}
The parallel transport equation is derived by 
\begin{equation*}
\partial_t\sigma(t,x)=-\Gamma^{W,x}(\sigma_t, \partial_t\rho_t)=\frac{1}{2}\big(\Delta_{\sigma_t}\Delta_{\rho_t}^{\dagger}\partial_t\rho_t+\Delta_{\partial_t\rho_t}\Delta_{\rho_t}^{\dagger}\sigma_t+\Delta_{\rho_t}(\nabla\Delta_\rho^{\dagger}\partial_t\rho_t\cdot\nabla\Delta_\rho^{\dagger}\sigma_t)\big).
\end{equation*}
And the geodesic equation is introduced by
\begin{equation*}
\partial_{tt}\rho_t=-\Gamma^{W,x}(\partial_t\rho_t,\partial_t\rho_t)=\Delta_{\partial_t\rho_t}\Delta_{\rho_t}^{\dagger}\partial_t\rho_t+\frac{1}{2}\Delta_{\rho_t}(\nabla \Delta_{\rho_t}^{\dagger}\partial_t\rho_t)^2.
\end{equation*}
\end{proof}
We introduce the curvature formulas in $(\mathcal{P}_+(M), W)$ following the derivation of \eqref{6main}. 
Denote $R_{W}(\cdot, \cdot)\cdot\colon C^{\infty}(T\mathcal{P}_+(M))\times  C^{\infty}(T\mathcal{P}_+(M))\times  C^{\infty}(T\mathcal{P}_+(M))\rightarrow  C^{\infty}(T\mathcal{P}_+(M))$.
\begin{proposition}[Curvature tensor]
Given $\sigma_1$, $\sigma_2$, $\sigma_3$, $\sigma_4\in T_\rho\mathcal{P}_+(M)$, then 
\begin{equation}\label{curvature}
\begin{split}
&g_W(R^W(\sigma_1,\sigma_2)\sigma_3, \sigma_4)\\
=&\frac{1}{4}\int_M \Big\{\sigma_2(x)\Delta_\rho^{\dagger} \Delta_{m(\sigma_1, \sigma_4)}\Delta_\rho^{\dagger}\sigma_3(x)+\sigma_1(x)\Delta_\rho^{\dagger} \Delta_{m(\sigma_2, \sigma_4)}\Delta_\rho^{\dagger}\sigma_3(x)\\
&\hspace{0.8cm}-\sigma_2(x)\Delta_\rho^{\dagger} \Delta_{m(\sigma_1, \sigma_3)}\Delta_\rho^{\dagger}\sigma_4(x)-\sigma_1(x)\Delta_\rho^{\dagger} \Delta_{m(\sigma_2, \sigma_3)}\Delta_\rho^{\dagger}\sigma_4(x)\\
&-2 n(\sigma_1, \sigma_2)\Delta_\rho^{\dagger}n(\sigma_3,\sigma_4)-n(\sigma_1,\sigma_3)\Delta_\rho^{\dagger}n(\sigma_2,\sigma_4)+n(\sigma_2,\sigma_3)\Delta_{\rho}^{\dagger}n(\sigma_1,\sigma_4)\Big\}dx,
\end{split}
\end{equation}
where operators $m$, $n\colon T_\rho\mathcal{P}_+(M)\times T_\rho\mathcal{P}_+(M)\rightarrow T_\rho\mathcal{P}_+(M)$ are defined by
\begin{equation*}
m(\sigma_a, \sigma_b):=-[\Delta_{\sigma_a}\Delta_{\rho}^{\dagger}\sigma_b+\Delta_{\sigma_b}\Delta_{\rho}^{\dagger}\sigma_a ]-\frac{1}{2}\Delta_\rho(\nabla\Delta_\rho^{\dagger}\sigma_a\cdot\nabla \Delta_\rho^{\dagger}\sigma_b),
\end{equation*}
and 
\begin{equation*}
n(\sigma_a, \sigma_b):=\Delta_{\sigma_a}\Delta_\rho^{\dagger}\sigma_b-\Delta_{\sigma_b}\Delta_{\rho}^{\dagger}\sigma_a.
\end{equation*}
\end{proposition}
We compute the Hessian operator in density manifold by the method in deriving \eqref{dH}. Denote $\textrm{Hess}_W(\cdot, \cdot)\colon C^{\infty}(T\mathcal{P}_+(M))\times C^{\infty}(T\mathcal{P}_+(M))\rightarrow C^{\infty}(\mathcal{P}_+(M))$.
\begin{proposition}[Hessian operator]
Given $\sigma_1,\sigma_2\in T_\rho\mathcal{P}_+(M)$, then
\begin{equation}\label{c-hess}
\begin{split}
\textrm{Hess}_{W}\mathcal{F}(\rho)(\sigma_1, \sigma_2)
=&\int_{M}\int_M \frac{\delta^2}{\delta\rho(x)\delta\rho(y)}\mathcal{F}(\rho)\sigma_1(x)\sigma_2(y)dxdy\\
+&\frac{1}{2}\int_M \frac{\delta}{\delta\rho(x)}\mathcal{F}(\rho) \big\{\Delta_{\sigma_1}\Delta_\rho^{\dd} \sigma_2(x)+\Delta_{\sigma_2}\Delta_{\rho}^{\dd}\sigma_1+\Delta_\rho(\nabla\Delta_\rho^{\dd}\sigma_1\cdot\nabla \Delta_\rho^{\dd}\sigma_2)\big\}dx,
\end{split}
\end{equation}

\end{proposition}

We next provide the formulation of the Laplace-Beltrami operator in density manifold, i.e. $\Delta_W\colon C^{\infty}(\mathcal{P}_+(M))\rightarrow C^{\infty}(\mathcal{P}_+(M))$, by the technique used in proposition \ref{p8}.

The following definitions are needed. Denote $\lambda_i(\rho)>0$, $i=1,2,\cdots$, be positive eigenvalues of $-\Delta_{\rho}$. I.e. there exists $u_i(x)\in C^{\infty}(M)$, such that
\begin{equation*}
-\nabla\cdot (\rho(x) \nabla u_i(x))=\lambda_i(\rho)u_i(x).
\end{equation*}
Let $\lambda_{\alpha,\mathcal{F}}(\rho)$, $\alpha\in I$ with the total index set $I$, be eigenvalues of operator $-\Delta_\rho\delta^2\mathcal{F}(\rho)$. I.e. there exists $v_\alpha(x)\in C^{\infty}(M)$, such that
\begin{equation*}
-\nabla_x\cdot\big(\rho(x)\nabla_x (\int_{M}  \frac{\delta^2}{\delta\rho(x)\delta\rho(y)}\mathcal{F}(\rho)v_\alpha(y)dy)\big)=\lambda_{\alpha,\mathcal{F}}v_{\alpha}(x).
\end{equation*}
\begin{proposition}[Laplace-Beltrami operator]
%The Laplace-Beltrami operator in $(\mathcal{P}_+(M), W)$ forms
Given $\mathcal{F}(\rho)\in C^{\infty}(\mathcal{P}_+(M))$, then
\begin{equation*}
\begin{split}
\Delta_W\mathcal{F}(\rho):=&tr_{L^2}((-\Delta_\rho)\delta^2\mathcal{F}(\rho))-\frac{1}{2}\int_{M}\nabla\log(\textrm{det}(-\Delta_\rho))(x)\nabla\frac{\delta}{\delta\rho(x)}\mathcal{F}(\rho)\rho(x)dx,
\end{split}
\end{equation*}
where $tr_{L^2}((-\Delta_\rho)\delta^2\mathcal{F}(\rho))=\sum_{\alpha\in I}\lambda_{\alpha,\mathcal{F}}(\rho)$,
and 
$\textrm{det}(-\Delta_\rho)=\Pi_{i=1}^{\infty}\lambda_i(\rho)$. 
\end{proposition}

Similar as the proof Lemma \ref{lemma}, it is straightforward to introduce the Jacobi equation on density manifold. 
\begin{proposition}[Variation of energy and Jacobi equation] 
Consider an infinitesmall deformation $\rho_\epsilon(t,x)=\rho(t,x)+\epsilon h(t,x)\in\mathcal{P}_+(M)$ with 
$h(t,x)\in C^{\infty}(M)$, $\int_M h(t,x)dx=0$, and $h(0,x)=h(1,x)=0$, 
\begin{equation*}
\mathcal{E}(\rho_\epsilon)=\int_{0}^1\frac{1}{2}\partial_t\rho_t^{\epsilon}(-\Delta_{\rho_t^\epsilon})^{\dd}\partial_t\rho^\epsilon_t dt=\mathcal{E}(\rho)+\epsilon\delta \mathcal{E}(\rho)(h)+\frac{\epsilon^2}{2}\delta^2\mathcal{E}(\rho)(h)+o(\epsilon^2).
\end{equation*}
Then the first and second variations satisfy 
\begin{equation*}
\delta \mathcal{E}(\rho)(h)=\int_0^1\int_M \partial_t\rho_t (-\Delta_{\rho_t})^{\dd}\big(\partial_th_t-\frac{1}{2}\Delta_{h_t}\Delta_{\rho_t}^{\dd}\partial_t\rho_t\big)dxdt ,
\end{equation*}
and 
\begin{equation*}
\delta^2 \mathcal{E}(\rho)(h)=\int_0^1\int_M\big(\partial_th_t-\Delta_{h_t}\Delta_{\rho_t}^{\dd}\partial_t\rho_t\big) (-\Delta_{\rho_t})^{\dd}\big(\partial_t h_t-\Delta_{h_t}\Delta_{\rho_t}^{\dd}\partial_t \rho_t\big)dxdt. 
\end{equation*}
Thus the Jacobi equation along the geodesic $\rho_t$ satisfies 
\begin{equation*}
\partial_t h_t-\Delta_{h_t}\Delta_{\rho_t}^{\dagger}\partial_t\rho_t=0,\quad h(0,x)=h(1,x)=0.
\end{equation*}
\end{proposition}
\begin{remark}
We notice that several other formulations of geometric operators, such as gradient operator, curvature tensor, etc in continuous space have been formulated in \cite{Gigli, OV, vil2008}. Here we emphasis their formulations in the tangent space, following the study of information geometry.  In addition, we emphasize that the general formulation of Hessian operator of energy functional in $L^2$--Wasserstein metric is new. 
\end{remark}
\begin{remark}
We notice that in continuous sample space, many explicit formulas here require smoothness assumptions on the measures.
 It is a delicate issue whether extensions to less regular settings hold \cite{Gigli}. We leave the study of the analysis for these geometric operators in the future work.
 \end{remark}
\subsection{Connections with Otto calculus}
In literature, the other coordinates, named Otto calculus \cite{vil2008}, in density manifold has been considered; see details in \cite{Lott}. These considerations are also studied in the Chapter 3 of \cite{Lafferty}. In this sequel, we illustrate the connection between Otto calculus and the ones in this paper. In other words, we simply formulate the calculus in density space by either tangent bundle or cotangent bundle. 

Denote the smooth cotangent space at $\rho\in \mathcal{P}_+(M)$ by $(T_\rho\mathcal{P}_+(M))^*$, i.e. 
\begin{equation*}
(T_\rho\mathcal{P}_+(M))^*=\{F_{\Phi},~\Phi\in C^{\infty}(M)\colon F_\Phi(\sigma)=\int_M\sigma(x)\Phi(x)dx,~\textrm{for any $\sigma\in T_\rho\mathcal{P}_+(M)$}\}.
\end{equation*}
For any constant $c\in \mathbb{R}$ and any $\sigma\in T_\rho\mathcal{P}_+(M)$,
\begin{equation*}
F_{(\Phi+c)}(\sigma)=\int_M\Phi(x)\sigma(x)dx+c\int_M\sigma(x)dx=\int_M\Phi(x)\sigma(x)dx=F_{\Phi}(\sigma),
\end{equation*}
I.e. $C^{\infty}(M)/\mathbb{R}\cong (T_\rho\mathcal{P}_+(M))^*$. 

We can identify the cotangent space and tangent space of density manifold by 
the map $\Phi\rightarrow V_{\Phi}=-\Delta_\rho \Phi(x)$. I.e. for any tangent vector $\sigma(x)\in T_\rho\mathcal{P}_+(M)$, there exists a unique $\Phi(x)\in C^{\infty}(M)/\mathbb{R}$, such that
 \begin{equation}\label{cot}
 \sigma(x)=V_{\Phi}(x)=-\nabla\cdot(\rho(x)\nabla \Phi(x)).
 \end{equation}
Following the property of elliptical operator $\Delta_\rho$, we have $C^{\infty}(M)/\mathbb{R}\cong T_\rho\mathcal{P}_+(M)$. Thus $(T_\rho\mathcal{P}_+(M))^*=T_\rho\mathcal{P}_+(M)$.

Thus the Riemannian inner product in density manifold can be represented the cotangent vectors. I.e. we apply potential function $\Phi(x)$ to represent the tangent vector $\sigma(x)\in T_\rho\mathcal{P}_+(M)$ using \eqref{cot}. Thus the inner product in density manifold can be formulated as 
\begin{equation*}
g_W(\sigma_1,\sigma_2)= \int_M \sigma_1(x)(-\Delta_\rho)^{\dd}\sigma_2(x)dx= \int_M \Phi_1(x)(-\Delta_\rho)\Phi_2(x)dx.
\end{equation*}
Following the Fr{\'e}chet manifold, all geometric formulas in density manifold, such as gradient, Hessian, geodesic, etc, derived in this paper does not depend on the coordinates system, so they are equivalent to the ones derived by Otto calculus. 
We next illustrate this equivalence by some formulas. 
\begin{proposition}[Geodesic by cotangent vectors]
Denote \begin{equation}\label{change}
\Phi_t=\Phi(t,x)=-(\nabla\cdot\rho(t,x)\nabla)^{\dd}\frac{\partial\rho}{\partial t}(t,x).
\end{equation} 
Then the geodesic equation \eqref{geo} is equivalent to the following two equations.  
One is the compressible Euler equation 
\begin{equation*}
\begin{cases}
&\partial_t\rho_t+\nabla\cdot(\rho_t\nabla\Phi_t)=0\\
&\partial_t\Phi_t+\frac{1}{2}\|\nabla\Phi_t\|^2=0.
\end{cases}
\end{equation*}
The other is (4.12) of \cite{Lafferty}, i.e. denote $v_t=v(t,x)=\nabla\Phi_t$, then
\begin{equation*}
\partial^2_t\rho_t=\nabla\cdot \big(\rho_t (\frac{\nabla\cdot(\rho_tv_t)}{\rho_t}v_t+\nabla_{v_t}v_t)\big).
\end{equation*}
\end{proposition}
\begin{remark}
Here for the geodesics formulation of $(\rho, \Phi)$, it is essentially well-known in literature, see \cite{Gigli, OV, vil2008}. We emphasize the fact that it is the Hamiltonian formulation of standard geodesics $\partial_{tt}\rho+\Gamma^W(\partial_t\rho, \partial_t\rho)=0$, where $\Phi$ is the momentum variable. 
\end{remark}
\begin{proof}
On the one hand, substituting \eqref{change} into \eqref{geo}, we have \begin{equation*}
\begin{split}
0=&\partial_t(\partial_t\rho_t)-\Delta_{\partial_t\rho_t}\Delta_{\rho_t}^{\dagger}\partial_t\rho_t-\frac{1}{2}\Delta_{\rho_t}(\nabla \Delta_{\rho_t}^{\dagger}\partial_t\rho_t)^2\\
=&-\partial_t(\Delta_{\rho_t}\Phi_t)+\Delta_{\partial_t\rho_t}\Phi_t-\frac{1}{2}\Delta_{\rho_t}(\nabla\Phi_t)^2\\
=&-\Delta_{\rho_t}(\partial_t\Phi_t+\frac{1}{2}(\nabla\Phi_t)^2).
\end{split}
\end{equation*}
Recall $\partial_t\rho_t=-\Delta_{\rho_t}\Phi_t$. We derive the compressible Euler equation, where $\Phi_t$ is unique up to a shrift of constant function w.r.t. $t$.

On the other hand, denote $v_t=v(t,x)=\nabla\Phi(t,x)$ in \eqref{change}, then $\nabla_{v_t}v_t=\frac{1}{2}\nabla(\nabla\Phi_t)^2$. Thus the geodesic equation \eqref{geo} forms
\begin{equation*}
\begin{split}
\partial_t(\partial_t\rho_t)=&\Delta_{\partial_t\rho_t}\Delta_{\rho_t}^{\dagger}\partial_t\rho_t+\frac{1}{2}\Delta_{\rho_t}(\nabla \Delta_{\rho_t}^{\dagger}\partial_t\rho_t)^2\\
=&-\nabla\cdot(\partial_t\rho_t\nabla\Phi_t)+\frac{1}{2}\nabla\cdot(\rho_t\nabla (-\nabla\Phi_t)^2)\\
=&\nabla\cdot(\nabla\cdot(\rho_t\nabla\Phi_t)\nabla\Phi_t)+\nabla\cdot(\rho_t\nabla_{v_t}v_t)\\
=&\nabla\cdot \big(\rho_t (\frac{\nabla\cdot(\rho_tv_t)}{\rho_t}v_t+\nabla_{v_t}v_t)\big).
\end{split}
\end{equation*}
\end{proof}

The second example is to derive the Hessian operator in the density manifold. 
\begin{proposition}[Hessian operator by cotangent vectors]\label{Hess_new}
\begin{equation*}
\begin{split}
\textrm{Hess}_{W}\mathcal{F}(\rho)\langle V_{\Phi_1}, V_{\Phi_2}\rangle:=&\int_{M}\int_M\nabla_x\nabla_y \frac{\delta^2}{\delta\rho(x)\delta\rho(y)}\mathcal{F}(\rho)\nabla\Phi_1(x)\nabla\Phi_2(y)\rho(x)\rho(y)dxdy\\
+&\int_M (\nabla^2_x\frac{\delta}{\delta \rho(x)}\mathcal{F}(\rho) \nabla\Phi_1(x),\nabla\Phi_2(x))\rho(x)dx,
\end{split}
\end{equation*}
where $\nabla^2_x$ denotes the Hessian operator in $M$ w.r.t. $x$.
\end{proposition}
\begin{proof}
From \eqref{c-hess} and $\sigma_1(x)=V_{\Phi_1}$, $\sigma_2(x)=V_{\Phi_2}$, we have
\begin{equation*}\begin{split}
&\textrm{Hess}_{W}\mathcal{F}(\rho)(V_{\Phi_1}, V_{\Phi_2})\\
=&\int_{M}\int_M \frac{\delta^2}{\delta\rho(x)\delta\rho(y)}\mathcal{F}(\rho)\nabla\cdot(\rho(x)\nabla\Phi_1(x))\nabla\cdot(\rho(y)\nabla\Phi_2(y))dxdy \hspace{2.35cm} (H1)\\
+&\frac{1}{2}\int_M \frac{\delta}{\delta\rho(x)}\mathcal{F}(\rho) \big\{-\Delta_{\sigma_1}\Phi_2(x)-\Delta_{\sigma_2}\Phi_1(x)+\Delta_\rho(\nabla\Phi_1(x)\cdot\nabla\Phi_2(x))\big\}dx\hspace{1.1cm}(H2)
\end{split}
\end{equation*}
We apply the following two steps to estimate (H1) and (H2). First, by integration by parts w.r.t. $x$ and $y$ twice, we derive
\begin{equation*}
(H1)=\int_{M}\int_M \nabla_x\nabla_y\frac{\delta^2}{\delta\rho(x)\delta\rho(y)}\mathcal{F}(\rho)\nabla_x\Phi_1(x)\nabla_y\Phi_2(y)\rho(x)\rho(y)dxdy.
\end{equation*}
Second, denote $F(x)=\frac{\delta}{\delta\rho(x)}\mathcal{F}(\rho)$ in (T2), then
\begin{equation*}
\begin{split}
(H2)=&-\frac{1}{2}\int_M F(x) \nabla\cdot\big\{\sigma_1(x)\nabla\Phi_2(x)+\sigma_2(x)\nabla\Phi_1(x)-\rho(x)\nabla(\nabla\Phi_1(x)\cdot\nabla\Phi_2(x))\big\}dx\\
=&\quad\frac{1}{2}\int_M \nabla F(x) \big\{\sigma_1(x)\nabla\Phi_2(x)+\sigma_2(x)\nabla\Phi_1(x)-\rho(x)\nabla(\nabla\Phi_1(x)\cdot\nabla\Phi_2(x)) \big\}dx\\
=&\quad\frac{1}{2}\int_M \nabla F(x) \big(\sigma_1(x)\nabla\Phi_2(x)+\sigma_2(x)\nabla\Phi_1(x)\big)dx\hspace{2.7cm} (H21)\\
&-\frac{1}{2}\int_{M}\rho(x)\nabla F(x)\cdot\nabla(\nabla\Phi_1(x)\cdot\nabla\Phi_2(x))\big\}dx \hspace{3.2cm} (H22) 
\end{split}
\end{equation*}
We next derive (H21). Substitute $\sigma_1(x)=-\nabla\cdot(\rho(x)\nabla\Phi_1(x))$ and $\sigma_2(x)=-\nabla\cdot(\rho(x)\nabla\Phi_2(x))$ into the above formula, then 
\begin{equation*}
\begin{split}
(H21)=&-\frac{1}{2}\int_M \nabla F(x)\big(\nabla\cdot(\rho(x)\nabla\Phi_1(x))\nabla\Phi_2(x)+\nabla\cdot(\rho(x)\nabla\Phi_2(x))\nabla\Phi_1(x)\big)dx\\
=&\frac{1}{2}\int_M\big( \nabla\Phi_1(x) \nabla(\nabla F(x)\cdot\nabla\Phi_2(x))+   \nabla\Phi_2(x) \nabla(\nabla F(x)\cdot\nabla\Phi_1(x))  \big)\rho(x)dx.
\end{split}
\end{equation*}

We last prove the following claim. 

\noindent\textbf{Claim 5}:
\begin{equation}\label{claim}
\begin{split}
&2(\nabla\nabla F(x)\nabla\Phi_1(x),\nabla\Phi_2(x))\\
=&\nabla\Phi_1(x) \nabla(\nabla F(x)\cdot\nabla\Phi_2(x))+   \nabla\Phi_2(x) \nabla(\nabla F(x)\cdot\nabla\Phi_1(x))-\nabla F(x)\nabla(\nabla\Phi_1(x)\cdot\nabla\Phi_2(x)).
\end{split}
\end{equation}
\begin{proof}[Proof of Claim]
\begin{equation*}
\begin{split}
&\sum_{1\leq a,b\leq d}\nabla_a(\nabla_bF\cdot \nabla^b\Phi_1)\nabla^a\Phi_2+\nabla_a(\nabla_bF\cdot \nabla^b\Phi_2)\nabla^a\Phi_1-\nabla_a(\nabla_b\Phi_1\cdot \nabla^b\Phi_2)\nabla^aF\\
=&\sum_{1\leq a,b\leq d}\nabla_a\nabla_bF\nabla^a\Phi_1\nabla^b\Phi_2+\nabla_bF\nabla_a\nabla^b\Phi_1\nabla^a\Phi_2+\nabla_a\nabla_bF\nabla^a\Phi_1\nabla^b\Phi_2\\
&\quad\quad+\nabla_a\nabla^b\Phi_2\nabla^a\Phi_1-\nabla^b(\nabla_a\Phi_1\cdot\nabla^a\Phi_2)\\
=&\sum_{1\leq a,b\leq d}2\nabla_a\nabla_bF\nabla^a\Phi_1\nabla^b\Phi_2+\nabla_bF\{\nabla_a\nabla^b\Phi_1\nabla^a\Phi_2+\nabla_a\nabla^b\Phi_2\nabla^a\Phi_1-\nabla^b(\nabla_a\Phi_1\cdot\nabla^a\Phi_2)\}\\
=&2\sum_{1\leq a,b\leq d}\nabla_a\nabla_bF\nabla^a\Phi_1\nabla^b\Phi_2.
\end{split}
\end{equation*}
\end{proof}
Using the fact $\textrm{Hess}_{W}\mathcal{F}(\rho)(V_{\Phi_1}, V_{\Phi_2})=(H1)+(H2)=(H1)+(H21)+(H22)$ and the claim, we finish the proof. 
\end{proof}
%It is worth mentioning that $(H21)$ connects to formula (6) in \cite{BE}.
It is worth mentioning that $\eqref{claim}$ is exactly the formula (6) in Bakry-{\'E}mery Gamma calculus paper \cite{BE}. This means that writing the Wasserstein Christoffel symbol from tangent to cotangent bundle recovers the famous iterative Gamma one conditions. Following this angle, we further illustrate the connection between information geometry and Wasserstein geometry in the next subsection. 
\subsection{Transport information geometry and Bakry-{\'E}mery $\Gamma_2$ operator}
In this sequel, we illustrate an identity between the Bakry-{\'E}mery $\Gamma_2$ operator in $M$ and the weighted Laplacian operator $\Delta_\rho$ in $\mathcal{P}_+(M)$. Here the identity serves the bridge for relating the Fisher-Rao geometry and Wasserstein geometry.  

%$\mathcal{P}_+(M)$ and its base manifold $M$.
We brief review the Bakry-{\'E}mery Gamma calculus. 
Denote $h(x)\in C^{\infty}(M)$ and the operator $L_h\colon C^{\infty}(M)\rightarrow C^{\infty}(M)$ by 
\begin{equation*}
L_h\Phi=\Delta\Phi(x)-\nabla h(x)\cdot \nabla\Phi(x).
\end{equation*}
Let $\Phi_1(x)$, $\Phi_2(x)\in C^{\infty}(M)$. The $\Gamma$ operator is given by
\begin{equation*}
\Gamma(\Phi_1,\Phi_2)=\frac{1}{2}[L_h(\Phi_1\Phi_2)-\Phi_1L_h\Phi_2-\Phi_2L_h\Phi_1]=\nabla\Phi_1\cdot\nabla\Phi_2,
\end{equation*}
and the $\Gamma_2$ operator is introduced by 
\begin{equation*}
\Gamma_2(\Phi_1,\Phi_2)=\frac{1}{2}[L_h\Gamma(\Phi_1,\Phi_2)-\Gamma(L_h\Phi_1, \Phi_2)-\Gamma(L_h\Phi_2,\Phi_1)].
\end{equation*}
Denote the Gibbs measure $\rho^*(x)=\frac{1}{K}e^{-h(x)}$, where $K=\int_{M}e^{-h(x)}dx$. Consider the relative entropy 
\begin{equation*}
\mathcal{H}(\rho|\rho^*)=\int_M\rho(x)\log\frac{\rho(x)}{\rho^*(x)}dx=\int_M\rho(x)\log\rho(x)dx+\int_{M}\rho(x)h(x)dx+K.
\end{equation*}
We demonstrate that the Hessian formula of $\mathcal{H}(\rho|\rho^*)$ in density manifold gives the following identity in Riemannian manifold $M$. 
\begin{proposition}\label{col}
\begin{equation*}
\begin{split}
\textrm{Hess}_W\mathcal{H}(\rho|\rho^*)(V_{\Phi_1}, V_{\Phi_2})=&\int_{M}\frac{\Delta_\rho\Phi_1\Delta_\rho\Phi_2}{\rho}+\big(\nabla\nabla\log\frac{\rho}{e^{-h}}\nabla\Phi_1,\nabla\Phi_2\big)\rho dx\\
=&\int_M \Gamma_2(\Phi_1,\Phi_2)\rho dx.
\end{split}
\end{equation*}
In particular, let $\Phi(x)=\Phi_1(x)=\Phi_2(x)$ and $h(x)\equiv 0$, then
\begin{equation*}\begin{split}
&\int_{M}\frac{1}{\rho(x)}\big(\nabla\cdot(\rho(x)\nabla\Phi(x))\big)^2+(\nabla\nabla\log{\rho(x)} \nabla\Phi(x), \nabla\Phi(x))\rho(x)dx\\
=&\int_M \big[\textrm{Ric}_M(\nabla\Phi(x), \nabla\Phi(x))+ \textrm{tr}(\nabla\nabla\Phi(x) \nabla\nabla\Phi(x))\big]\rho(x)dx,
\end{split}
\end{equation*}
where $\textrm{Ric}_M$ denotes the Ricci tensor on $M$.
\end{proposition}
\begin{proof}
The proof is straightforward by the Hessian operator of relative entropy in density manifold. Since \begin{equation}\label{d-1}
\begin{split}
\textrm{Hess}_W\mathcal{H}(\rho|\rho^*)(V_{\Phi_1}, V_{\Phi_2})%&\int_{M}\frac{\Delta_\rho\Phi_1\Delta_\rho\Phi_2}{\rho}+\big(\nabla\nabla\log\frac{\rho}{e^{-h}}\nabla\Phi_1,\nabla\Phi_2\big)\rho dx\\
=&\int_{M}\frac{\Delta_\rho\Phi_1\Delta_\rho\Phi_2}{\rho}+\big(\nabla\nabla\log{\rho}\nabla\Phi_1,\nabla\Phi_2\big)\rho+\big(\nabla\nabla h\nabla\Phi_1,\nabla\Phi_2\big)\rho dx.
\end{split}
\end{equation}
Rewriting \eqref{claim} with $\Gamma$ operator and doing integration by parts, we have
\begin{equation}\label{d-2}
\begin{split}
&\int_M(\nabla\nabla\log{\rho}\nabla\Phi_1,\nabla\Phi_2)\rho dx\\
=&\frac{1}{2}\int_M\big[\Gamma(\Gamma(\log{\rho}, \Phi_1), \Phi_2)+\Gamma(\Gamma(\log\rho, \Phi_2), \Phi_1)-\Gamma(\Gamma(\Phi_1, \Phi_2), \log{\rho})\big]\rho dx\\
=&-\frac{1}{2}\int_M\Gamma(\log\rho, \Phi_1)\Delta_\rho\Phi_2+\Gamma(\log\rho, \Phi_2)\Delta_\rho\Phi_1- \Gamma(\log\rho, \Gamma(\Phi_1,\Phi_2))\rho  dx.
\end{split}
\end{equation}
Notice 
\begin{equation}\label{d-31}
\nabla\log\rho=\frac{1}{\rho}\nabla\rho,
\end{equation} 
then 
\begin{equation}\label{d-3}
\frac{\Delta_\rho\Phi}{\rho}=\frac{\nabla\cdot(\rho\nabla\Phi)}{\rho}=\frac{\Gamma(\rho,\Phi)+\rho\Delta\Phi}{\rho}=\Gamma(\log\rho, \Phi)+\Delta\Phi.
\end{equation}
Substituting \eqref{claim}, \eqref{d-2}, \eqref{d-31}, \eqref{d-3} into \eqref{d-1}, we have
\begin{equation*}
\begin{split}
&\textrm{Hess}_W\mathcal{H}(\rho|\rho^*)(V_{\Phi_1}, V_{\Phi_2})\\
=&\frac{1}{2}\int_{M}(\frac{\Delta_\rho\Phi_1}{\rho}-\Gamma(\log\rho, \Phi_1))\Delta_{\rho}\Phi_2+(\frac{\Delta_\rho\Phi_1}{\rho}-\Gamma(\log\rho, \Phi_2))\Delta_{\rho}\Phi_1-\Gamma(\log\rho, \Gamma(\Phi_1,\Phi_2))\rho dx\\
&+\int_M\big(\nabla\nabla h\nabla\Phi_1,\nabla\Phi_2\big)\rho dx\\
=&\frac{1}{2}\int_{M}\Big\{\Delta\Phi_1\Delta_\rho\Phi_2+\Delta\Phi_2\Delta_{\rho}\Phi_1-\Gamma(\rho, \Gamma(\Phi_1,\Phi_2))\Big\}+\Big\{(\nabla\nabla h\nabla\Phi_1,\nabla\Phi_2)\Big\}\rho dx\\
=&\frac{1}{2}\int_M\Big\{\Delta\Gamma(\Phi_1,\Phi_2)-\Gamma(\Delta\Phi_1,\Phi_2)-\Gamma(\Delta\Phi_2,\Phi_1)\Big\}\rho\\
&\quad+\Big\{\Gamma(\Phi_1, \Gamma(V,\Phi_2))+\Gamma(\Phi_2, \Gamma(V,\Phi_1))-\Gamma(V, \Gamma(\Phi_1,\Phi_2)) \Big\}\rho dx\\
=&\frac{1}{2}\int_M\Big\{(\Delta-\nabla h\cdot\nabla)\Gamma(\Phi_1,\Phi_2)-\Gamma((\Delta-\nabla h\cdot \nabla)\Phi_1,\Phi_2)-\Gamma((\Delta-\nabla h\cdot\nabla)\Phi_2,\Phi_1)\Big\}\rho dx\\
=&\int_M\Gamma_2(\Phi_1,\Phi_2)\rho dx,
\end{split}
\end{equation*}
where the last equality is from the definition of $\Gamma_2$ operator. As in (4a) of \cite{BE}, from the Bochner's formula, we have \begin{equation*}
\begin{split}
\Gamma_2(\Phi_1,\Phi_2)=&(\nabla\nabla h\nabla\Phi_1,\nabla\Phi_2)+\textrm{Ric}_M(\nabla\Phi_1, \nabla\Phi_2)+tr(\nabla\nabla\Phi_1 \nabla\nabla\Phi_2),
\end{split}
\end{equation*}
which finishes the proof.
\end{proof}
\begin{remark}
It is worth mentioning that the connection between Wasserstein Hessian and Gamma two operators have been observed in \cite{OV, vil2008}. Our contribution is to reformulate it into the tangent space and observe the relation of connecting tangent and cotangent bundles. In particular, we notice that we discover the relation of Bochner's formula with the Fisher-Rao metric. Following this viewpoint, the generalization of Gamma two operators becomes feasible. This needs to follow the combination angle of information geometry and transport geometry presented in this work. We leave details of their formulations in future work. 
\end{remark}
\begin{remark}
We note that if $\rho(x)\equiv 1$ and $h(x)\equiv 0$, then Proposition \ref{col} shows the standard Yano's formula \cite{Yano}, i.e.
\begin{equation*}
\int_{M}\big(\nabla\cdot(\nabla\Phi(x))\big)^2dx=\int_M \textrm{Ric}_M(\nabla\Phi(x), \nabla\Phi(x))+ \textrm{tr}(\nabla\nabla\Phi(x) \nabla\nabla\Phi(x))dx.
\end{equation*}
The derivation of Yano's formula by Hessian operator in density manifold is reported in \cite{li-theory, li-thesis}, which is one of the motivations for this paper.
This fact will further generalize various Poincar{\'e} inequalities. We leave details studies of this point in future work. 
\end{remark}
\section{Differential equations in probability manifold}\label{section4}
In this section, we use several examples to illustrate the geometry formulas in $\mathcal{P}_+(G)$. 

\begin{example}[Nonlinear metric tensor]
Consider a three vertices graph with $V=\{1,2,3\}$ and $\omega_{12}=\omega_{23}=1$, i.e. \begin{tikzpicture}[->,shorten >=1pt,auto,node distance=2cm,
        thick,main node/.style={circle,fill=blue!20,draw,minimum size=0.5cm,inner sep=0pt]} ]
   \node[main node] (1) {1};
    \node[main node] (2) [right of=1]  {2};
    \node[main node](3)[right of=2]{3};
    \path[-]
    (1) edge node {} (2)
    (2) edge node{} (3);
\end{tikzpicture}.
The probability simplex supported on the above graph is a triangular in $\mathbb{R}^3$, which is a two dimensional manifold.
\begin{equation*}
\mathcal{P}(G)=\{(\rho_i)_{i=1}^3\in \mathbb{R}^3~:~\sum_{i=1}^3\rho_i=1, \quad \rho_i\geq 0\}.
\end{equation*}
 The $L^2$ Wasserstein metric on this graph solves the variational problem \eqref{w2}:
\begin{equation*}
\inf_{\Phi(t)\colon \rho(0)=\rho^0,~\rho(1)=\rho^1}\int_0^1 (\Phi_1(t)-\Phi_2(t))^2\frac{\rho_1(t)+\rho_2(t)}{2}+(\Phi_2(t)-\Phi_3(t))^2\frac{\rho_2(t)+\rho_3(t)}{2}  dt
\end{equation*}
Let $\rho_2=1-\rho_1-\rho_3$ and $\dot\rho=L(\rho)\Phi$, i.e. $\dot\rho_1=(\Phi_1-\Phi_2)\frac{\rho_1+\rho_2}{2}$, $\dot\rho_3=(\Phi_3-\Phi_2)\frac{\rho_2+\rho_3}{2}$, then \eqref{w2} can be reformulated as \begin{equation*}
\inf_{\rho(t)\in \mathcal{P}_{+}(G)\colon \rho(0)=\rho^0,~\rho(1)=\rho^1}\int_0^1\frac{\dot\rho_1(t)^2}{1-\rho_3(t)}+\frac{\dot\rho_3(t)^2}{1-\rho_1(t)}dt.
\end{equation*}
By solving the above problem numerically, a geodesic triangular connecting three discrete probabilities in $\mathcal{P}_+(G)$ is provided. It can be seen that $(\mathcal{P}_+(G), g_W)$ is with a nonlinear metric tensor and the geodesics depend on the structure of graph. 
\begin{figure}[h]
%\vspace{-0.5cm}
\centering
\includegraphics[width=0.4\textwidth]{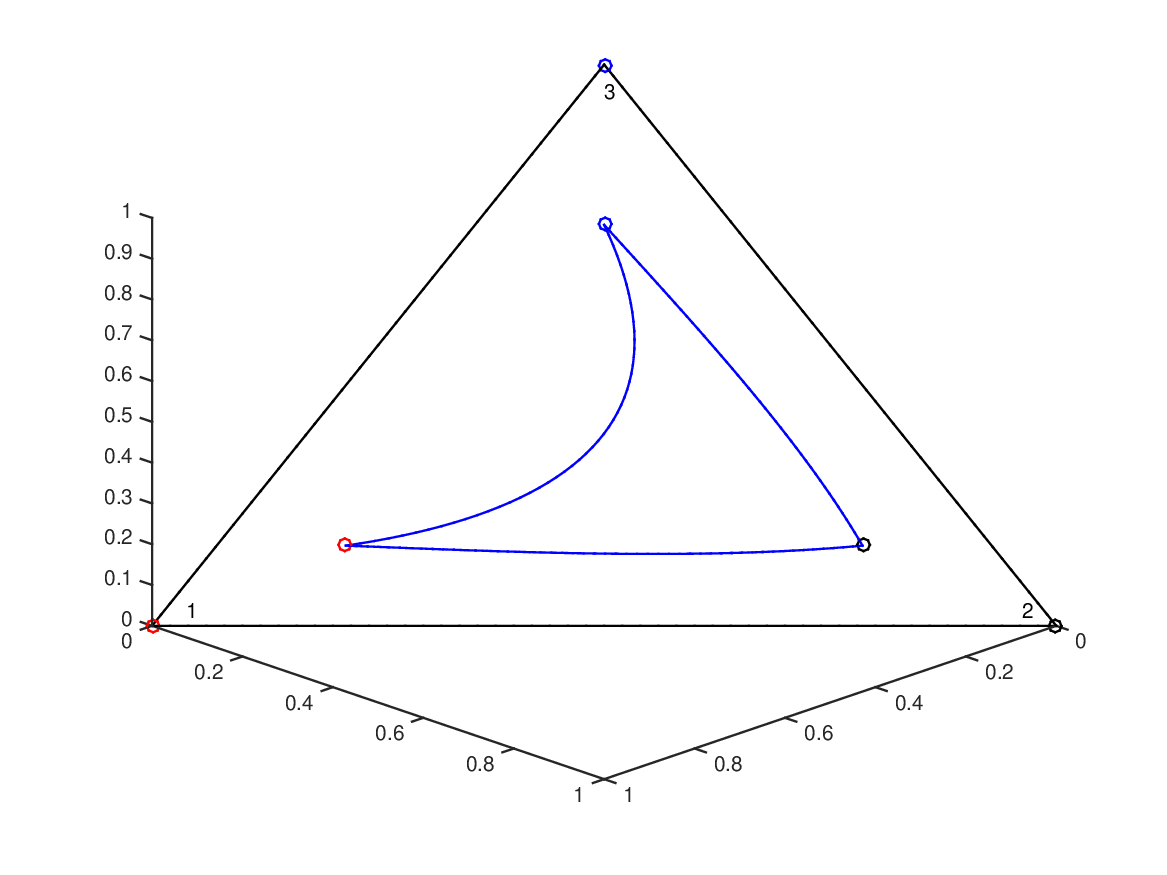}
\end{figure}
%\vspace{-0.5cm}
\end{example}

\begin{example}[Ordinary differential equations in probability manifold]
Consider $\mathcal{F}(\rho)\in C^{\infty}(\mathcal{P}(G))$. On one hand, the gradient flow of $\mathcal{F}(\rho)$ gives 
\begin{equation*}
\dot\rho=-\textrm{grad}_W\mathcal{F}(\rho),
\end{equation*}
i.e. 
\begin{equation}\label{gradient}
\dot\rho=-L(\rho)d_\rho\mathcal{F}(\rho)=\textrm{div}(\rho\nabla_Gd_\rho\mathcal{F}(\rho)).
\end{equation}
 On the other hand, the Hamiltonian flow of $\mathcal{F}(\rho)$ refers $$\nabla_{\dot \rho}\dot\rho=\ddot\rho+(\dot\rho^{\ts} \Gamma^k\dot\rho)_{k=1}^n=-d_\rho\mathcal{F}(\rho),$$ 
i.e.
\begin{equation}\label{Hamiltonian}
\ddot\rho-L(\dot\rho)L(\rho)^{\dagger}\dot\rho+\frac{1}{2}L(\rho)(\nabla_GL(\rho)^{\dagger}\dot\rho\circ\nabla_GL(\rho)^{\dagger}\dot\rho)=-d_\rho\mathcal{F}(\rho).
\end{equation}
Equation \eqref{Hamiltonian} can be rewrite as the first order ODE system. Consider the inner coordinate (Legendre transform), i.e. $\dot\rho=L(\rho)\Phi$, then \eqref{Hamiltonian} forms \begin{equation*}
\dot\rho+\textrm{div}(\rho\nabla_G\Phi)=0,\quad \dot\Phi+\frac{1}{2}\nabla_G\Phi\circ\nabla_G\Phi=-d_\rho\mathcal{F}(\rho),
\end{equation*}
which can also be denoted as 
\begin{equation*}\label{g_ham}
 \dot\rho=\frac{\partial}{\partial \Phi}\mathcal{H}(\rho,\Phi),\quad \dot\Phi=-\frac{\partial}{\partial\rho}\mathcal{H}(\rho,\Phi),
  \end{equation*}
 where 
\begin{equation*}
 \mathcal{H}(\rho, \Phi):=\frac{1}{2}(\nabla_G \Phi, \nabla_G \Phi)_\rho+\mathcal{F}(\rho).
\end{equation*}
Several examples of \eqref{gradient}, \eqref{Hamiltonian} are studied in \cite{li-SE, li-theory}. For studying these ODEs' dynamical properties, the derived Hessian operators, curvature tensors and Jacobi equations in $(\mathcal{P}_+(G), g_W)$ are needed. We will work on them in this series of work.
\end{example}

In addition, we are curious about the drift diffusion process associated with the canonical volume form in $(\mathcal{P}_+(G), g_W)$. 
\begin{example}[Stochastic differential equations in probability manifold]
Consider the Fokker-Planck equation in $(\mathcal{P}_+(G), g_W)$ with drift vector $\nabla_W\mathcal{F}(\rho)$, $\mathcal{F}(\rho)\in C^{\infty}(\mathcal{P}(G))$, diffusion constant $\beta>0$ on a compact set $\mathcal{B}\subset \mathcal{P}_+(G)$. The zero-flux condition is proposed on $\partial \mathcal{B}$. 
Then the Fokker-Planck equation $\frac{\partial \mathbb{P}(t,\rho)}{\partial t}=\textrm{div}_{W}(\mathbb{P}(t,\rho)\nabla_{W}\mathcal{F}(\rho))+\beta\Delta_{W}\mathbb{P}(t,\rho)$ satisfies 
\begin{equation}\label{FPE}
\frac{\partial \mathbb{P}(t,\rho)}{\partial t}=\Pi(\rho)^{\frac{1}{2}}\nabla_\rho\cdot(L(\rho)\big(\mathbb{P}(t,\rho)d_\rho\mathcal{F}(\rho)+\beta d_\rho\mathbb{P}(t,\rho)\big) \Pi(\rho)^{-\frac{1}{2}}).
\end{equation}
The associated drift-diffusion process on $(\mathcal{P}_+(G), g_W)$ satisfies
\begin{equation}\label{SDE}
d\rho_t=-L(\rho_t)(d_{\rho_t}\mathcal{F}(\rho_t)+\frac{\beta}{2} d_{\rho_t}\log\Pi(\rho_t))dt+\sqrt{2\beta}{L(\rho_t)}^{\frac{1}{2}}dB_t.
\end{equation}
The stationary solution of \eqref{FPE} is a Gibbs measure in the probability set over manifold $(\mathcal{P}_+(G), g_W)$,
\begin{equation}\label{WGibbs}
\mathbb{P}^*(\rho)=\frac{1}{K}e^{-\frac{\mathcal{F}(\rho)}{\beta}},\quad \textrm{where $K=\int_{\mathcal{B}}\Pi(\rho)^{-\frac{1}{2}}e^{-\frac{\mathcal{F}(\rho)}{\beta}}d\textrm{vol}$}.
\end{equation}
\begin{proof}
Using the fact that 
\begin{equation*}
d_\rho\mathbb{P}(t,\rho)=\mathbb{P}(t,\rho)d_\rho\log\mathbb{P}(t,\rho),
\end{equation*}
then \eqref{FPE} forms
\begin{equation}\label{FPE1}
\frac{\partial \mathbb{P}(t,\rho)}{\partial t}=\Pi(\rho)^{\frac{1}{2}}\nabla_\rho\cdot(L(\rho)\mathbb{P}(t,\rho)\big(d_\rho\mathcal{F}(\rho)+ \beta d_\rho\log\mathbb{P}(t,\rho)\big) \Pi(\rho)^{-\frac{1}{2}}).
\end{equation}
Denote the density function in Euclidean volume form, $f(t,\rho)=\mathbb{P}(t,\rho) \Pi(\rho)^{-\frac{1}{2}}$, then \eqref{FPE1} forms
\begin{equation*}
\frac{\partial f(t,\rho)}{\partial t}=\nabla_\rho\cdot [f(t,\rho) L(\rho) d_\rho\mathcal{F}(\rho)]+\beta\nabla_\rho\cdot [f(t,\rho) L(\rho)d_\rho\log \frac{f(t,\rho)}{\Pi(\rho)^{-\frac{1}{2}}}].
\end{equation*}
Thus the infinitesimal generator of \eqref{FPE} is 
\begin{equation*}
A\hat{f}(\rho)=-\sum_{i=1}^n\frac{\partial \hat{f}(\rho)}{\partial \rho_i}\big(L(\rho)(d_\rho\mathcal{F}(\rho)+\frac{\beta}{2}d_\rho\log\Pi(\rho)\big)(i)+\beta\sum_{1\leq i,j\leq n} (L(\rho)^{\frac{1}{2}}L(\rho)^{\frac{1}{2}})_{ij}\frac{\partial^2\hat{f}(\rho)}{\partial\rho_i\partial\rho_j},
\end{equation*}
for any compactly-supported $C^2$ function $\hat{f}(\rho)$ with $\rho\in \mathcal{B}$. Then the stochastic differential equation \eqref{SDE} is derived. 
By solving $\frac{\partial}{\partial t}f(t,\rho)=0$, we have
$$f^*(\rho)=\frac{1}{K}\Pi(\rho)^{-\frac{1}{2}}e^{-\frac{\mathcal{F}(\rho)}{\beta}},\quad \textrm{where $K=\int_{\mathcal{B}}\Pi(\rho)^{-\frac{1}{2}}e^{-\frac{\mathcal{F}(\rho)}{\beta}}d\textrm{vol}$},$$
i.e. $\mathbb{P}^*(\rho)=f^*(\rho)\Pi(\rho)^{\frac{1}{2}}$ satisfies \eqref{WGibbs}. 
\end{proof}
\begin{remark}
Formula \eqref{SDE} suggests that the drift-diffusion process in $(\mathcal{P}_+(M), W)$ forms 
\begin{equation*}
d\rho(t,x)=\nabla\cdot\big(\rho(t,x)\nabla\frac{\delta}{\delta\rho}(\mathcal{F}+\frac{\beta}{2}\log\textrm{det}(-\Delta_\rho))(t,x)\big)dt+\sqrt{2\beta}(-\nabla\cdot(\rho(t,x)\nabla))^{\frac{1}{2}}\cdot dB(t,x),
\end{equation*}
where $B(t,x)$ is the standard space-time Brownian motion. 
\end{remark}
\end{example}
\section{Discussion}\label{section7}
In this paper, we introduce the geometry formulas in probability manifold on graphs with $L^2$-Wasserstein metric. 
The main idea is to propose a Riemannian metric in the positive measure space, and study the probability manifold as its submanifold. Similar derivations have been introduced into the infinite-dimensional density manifold, whose base is a finite-dimensional Riemannian manifold.

The current study leaves many questions open both analytically and geometrically.  On the one hand, it is needed to understand the regularity issues behind the Riemannian metric defined on the positive measure space. 
In other words, given $\mu^0$, $\mu^1\in \mathcal{M}_+(M)$, consider a distance function $\textrm{Dist}_{\mathcal{M}}\colon \mathcal{M}_+(M)\times \mathcal{M}_+(M)\rightarrow \mathbb{R}$ by 
\begin{equation*}
\begin{split}
&\textrm{Dist}_{\mathcal{M}}(\mu_0, \mu_1)^2\\
=&\inf_{\mu_t\in\mathcal{M}_+(M),~\mu(0)=\mu^0,~\mu(1)=\mu^1}\big\{\int_0^1\int_{M}  \partial_t\mu_tg(\mu_t)\partial_t\mu_tdx dt\big\}\\
=&\inf_{\Phi_t\colon\mu_0=\mu^0,~\mu_1=\mu^1}\big\{\int_0^1\big[\int_{M} (\nabla\Phi_t)^2\mu_tdx+ (\int_{M} \Phi_tdx)^2\big]dt\colon \partial_t\mu_t+\nabla\cdot(\rho_t\nabla\Phi_t)=\int_M\Phi_tdx\big\},
\end{split}
\end{equation*}
where the second equality is from the change of variable $\partial_t\mu_t=g(\mu_t)^{-1}\Phi_t$, and $g(\mu_t)$ is defined by \eqref{gmu}. 

On the other hand, there are many interactions and connections between the geometry formulas in optimal transport density manifold, information geometry, and its base manifold $M$; see Proposition \ref{col} for the connection with Gamma two operators. We expect that the geometric study of density manifold would introduce generalized versions of Gamma two operators. These understandings will provide the insights for the ``geometry'' of finite graphs and statistical models in AI and inference problems \cite{Amari, divergence, IG2}, which are useful for modeling and computations.  In this series of work, we will continue to address these problems.

%%\textbf{Acknowledgement}
%%The author would thank Prof. Shui-Nee Chow, Haomin Zhou, Wilfrid Gangbo for the motivation of this paper.
%\bibliographystyle{abbrv}
%\bibliography{pm}
\end{document}